\documentclass[10pt]{amsart}
\usepackage{latexsym, amsmath, amssymb,amsthm,amsopn,amsfonts}
\usepackage{version}
\usepackage[all]{xy}
\usepackage{hyperref}
\usepackage{epsfig,graphics,color,graphicx,graphpap}
\usepackage{dsfont}

\newtheorem{theorem}{Theorem}
\newtheorem{corollary}[theorem]{Corollary}
\newtheorem{definition}[theorem]{Definition}

\newtheorem{lemma}[theorem]{Lemma}
\newtheorem{proposition}[theorem]{Proposition}

\numberwithin{equation}{section}
\numberwithin{theorem}{section}

\newcommand{\F}{{\mathcal F}}

\newcommand{\ZZ}{{\mathbb Z}}

\newcommand{\RR}{{\mathbb R}}

\newcommand{\NN}{{\mathbb N}}

\newcommand{\half}{\frac{1}{2}}

\newcommand{\supp}{\operatorname{supp}}

\renewcommand{\Im}{\mathop{\rm Im}\nolimits}

\theoremstyle{plain}

\newtheorem{prop}{Proposition}[section]

\newtheorem{rema}[prop]{Remark}
\newtheorem*{Notations}{Notations}
\theoremstyle{definition}

\numberwithin{equation}{section}

\def\squarebox#1{\hbox to #1{\hfill\vbox to #1{\vfill}}}

\newcommand{\ds}{\displaystyle}
\newcommand{\p}{\partial}
\usepackage{amsxtra}

\def\dH{\dot{H}}
\def\ZZ{{\mathbb Z}}
\def\NN{{\mathbb N}}
\def\reals{{\mathbb R}}

\def\Ci{{\mathcal C}^\infty}

\def\Im{\,\mathrm{Im}\,}

\def\WF{\mathrm{WF}_h}

\def\supp{\mathrm{supp}\,}

\def\id{\,\mathrm{id}\,}

\def\O{{\mathcal O}}
\def\SS{{\mathbb S}}
\def\s{{\mathcal S}}
\def\scl{{\mathcal{S}_{\text{cl}}}}

\def\Op{\mathrm{Op}\,}

\def\esssupp{\text{ess-supp}\,}

\def\phi{\varphi}
\def\half{{\frac{1}{2}}}
\def\dist{\text{dist}\,}

\def\be{\begin{eqnarray*}}
\def\ee{\end{eqnarray*}}
\def\ben{\begin{eqnarray}}
\def\een{\end{eqnarray}}

\def\lll{\left\langle}
\def\rrr{\right\rangle}

\def\L2R{L_{\text{Rest}}^2}

\def\11{\mathds{1}}

\def\RR{\mathbb{R}}
\def\tpsi{\tilde{\psi}}

\def\L2c{L^2_{\text{comp}}}
\def\tphi{\tilde{\phi}}

\def\tDelta{\widetilde{\Delta}}

\def\F{\mathcal{F}}

\def\dWF{\mathrm{WF}_{h,\delta, \gamma}}

\def\Vol{\text{Vol}}
\def\sWF{\mathrm{WF}}

\ifx\pdfoutput\undefined
  \DeclareGraphicsExtensions{.eps}
\else
  \ifx\pdfoutput\relax
    \DeclareGraphicsExtensions{.eps}
  \else
    \ifnum\pdfoutput>0
      \DeclareGraphicsExtensions{.pdf}
    \else
      \DeclareGraphicsExtensions{.eps}
    \fi
  \fi
\fi

\title
[Nonlinear quasimodes]
{Nonlinear quasimodes near elliptic periodic geodesics}

\author[P. Albin]
{Pierre Albin}
\email{albin@math.jussieu.fr}
\address{Institut de Math\'ematiques de Jussieu \\
175 rue du Chevaleret \\
Paris 75013, France} 

\author[H. Christianson]
{Hans Christianson}
\email{hans@math.unc.edu}
\address{Department of Mathematics, UNC-Chapel Hill \\ CB\#3250
  Phillips Hall \\ Chapel Hill, NC 27599}

\author[J.L. Marzuola]
{Jeremy L. Marzuola}
\email{marzuola@math.unc.edu}
\address{Department of Mathematics, UNC-Chapel Hill \\ CB\#3250
  Phillips Hall \\ Chapel Hill, NC 27599}

\author[L. Thomann]
{Laurent Thomann}
\email{Laurent.Thomann@univ-nantes.fr}
\address{ Universit\'e de Nantes, Laboratoire de Math\'ematiques Jean Leray \\
Nantes, France}

\begin{document}

\begin{abstract}
We consider the nonlinear Schr\"odinger equation on a compact
manifold near an elliptic periodic geodesic.  Using a geometric optics
construction, we construct quasimodes to a nonlinear stationary
problem which are highly localized near the periodic geodesic.  We
show the
nonlinear 
Schr\"odinger evolution of such a quasimode remains localized near the
geodesic, at least for short times.
\end{abstract}

\maketitle

\section{Introduction}
In this note, we study the Nonlinear Schr\"odinger Equation (NLS) on a
compact, Riemannian manifold with a periodic elliptic (or stable)
geodesic, which we define and discuss in more detail in Appendix \ref{quasi}.  Specifically, we study solutions to
\begin{eqnarray}
\label{nls}
\left\{ \begin{array}{c}
i \p_t \psi = \Delta_g \psi + \sigma |\psi|^{p} \psi, \\
\psi (x,0) = \psi_0,
\end{array} \right.
\end{eqnarray}
where $\Delta_g$ is the standard (negative semidefinite) Laplace-Beltrami operator and the solution is of the form
\begin{eqnarray*}
\psi (x,t) = e^{-i \lambda t} u(x).
\end{eqnarray*}
To solve the resulting nonlinear elliptic equation, which can be
analyzed using constrained variations, we will use Fermi
coordinates to construct a nonlinear quasimode similar to the one
presented in \cite{Tho} for an arbitrary manifold with a periodic,
elliptic orbit.  To do so, we must analyze the metric geometry in a
neighbourhood of a periodic orbit, for which we use the presentation in
\cite{Tubes}.  For further references to construction of quasimodes
along elliptic geodesics, see the seminal works Ralston \cite{Ral-beam},
Babi{\v{c}} \cite{Bab-geo}, Guillemin-Weinstein \cite{GuWe-geo}, Cardoso-Popov \cite{CaPo-quasi}, as well as the
thorough survey on spectral theory by Zelditch
\cite{Zel-survey}.  For experimental evidence of the existence of such
Gaussian beam type solutions in nonlinear optics as solutions to
nonlinear Maxwell's equations, see the recent paper by Schultheiss et
al \cite{optics}.

Henceforward, we assume that $(M,g)$ is  a compact Riemannian manifold
of dimension $d \geq 2$ without boundary, and that it admits an elliptic periodic geodesic $\Gamma \subset M$.  For    $\beta>0$, we introduce the notation 
\begin{equation*}
U_\beta= \{ x \in M : \dist_g(x, \Gamma) < \beta \}.
\end{equation*}
Let $L>0$ be the period of $\Gamma$.  Throughout the paper, each time we refer to the small parameter $h>0$, this means that $h$ takes the form
\begin{equation}\label{formeh}
 h=\frac L{2\pi N},
\end{equation}
for some large integer $N\in \NN$.

Our first result is the existence, in the case of a smooth nonlinearity, of    quasimodes which are highly
concentrated near the elliptic periodic geodesic $\Gamma$. More precisely : 
\begin{theorem}
\label{T:0}
Let $(M,g)$ be a compact Riemannian manifold of dimension $d \geq 2$, without boundary and  which has an elliptic periodic geodesic.  Let $p$ be an even integer,  let $s\geq 0$ and assume that 
\begin{equation*}
 p \big(\frac{d-1}4-s\big)<1.
\end{equation*}
Then for  $h \ll 1$ sufficiently small, and any $\delta>0$, there exists  
$\phi_h(x) \in H^{\infty}(M)$ satisfying 
\begin{itemize}
\item[(i)] Frequency localization : For all $r\geq 0$, there exist $C_{1}, C_{2}$ independent of $h$ such that 
\begin{equation*}
 C_{1} h^{s-r} \leq \| \phi_h \|_{H^r(M)} \leq C_{2} h^{s-r}.
\end{equation*}
\item[(ii)] Spatial localization near $\Gamma$ : 
\begin{equation*}
\|\phi_h \|_{H^s(M \setminus U_{h^{1/2-\delta}})} = \O(h^{\infty}) \| \phi_h \|_{H^s(M)}.
\end{equation*}
\item[(iii)] Nonlinear quasimode : There exists $\lambda(h) \in \RR$ so that $\phi_{h}$ satisfies the equation 
\begin{equation*}
-\Delta_{g} \phi_h=\lambda(h)\phi_{h}    -\sigma | \phi_h|^{p}\phi_h+\mathcal{O}(h^{\infty}).
\end{equation*}
\end{itemize}
\end{theorem}

 The quasimode $\phi_{h}$ and the nonlinear eigenvalue $\lambda(h)$ will be constructed thanks to a WKB method, this will give a precise description of these objects (see Section \ref{sec:nthorder}). For instance, $\lambda(h)$ reads $\lambda(h)=h^{-2}-E_{0}h^{-1}+o(h^{-1})$, for some $E_{0}\in \RR$.
 
As a consequence of Theorem \ref{T:0}, and under the same   assumptions, we can state : 
\begin{theorem}
\label{T:1}
Consider the function $\phi_{h}$ given by Theorem \ref{T:0}. There exists $c_{0}>0$, such that if we denote by $\ds T_{h}=c_{0}h^{p(\frac{d-1}4-s)}\ln(\frac1h)$, the solution $u_{h}$ of the Cauchy problem \eqref{nls} with initial condition $\phi_{h}$ satisfies 
 \[
\| u_h\|_{L^\infty ([0,T_h]; H^s(M \setminus
  U_{h^{1/2-\delta}})) } \leq C h^{(d+1)/4}.
\]
\end{theorem}

To construct a quasimode with $\mathcal{O}(h^{\infty})$ error, we need that the nonlinearity is of polynomial type, that is $p\in 2\NN$.     However, in  the special case where $d=2$ and $s=0$,  we are  able to obtain a result for any $0<p<4$, which is a weaker version of Theorem \ref{T:1}.

\begin{theorem}
\label{T:2}
Let $(M,g)$ be a compact Riemannian surface  without boundary which admits an elliptic periodic geodesic and  let $0<p<4$. Then for  $h \ll 1$ sufficiently small, and any $\delta,\epsilon>0$, there exist
$\phi_h(x) \in H^{\infty}(M)$ and $\nu>0$ satisfying 
\begin{itemize}
\item[(i)] Frequency localization : For all $r\geq 0$, there exist $C_{1}, C_{2}$ independent of $h$ such that 
\begin{equation*}
 C_{1} h^{-r} \leq \| \phi_h \|_{H^r(M)} \leq C_{2} h^{-r}.
\end{equation*}
\item[(ii)] Spatial localization near $\Gamma$ : 
\begin{equation*}
\|\phi_h \|_{L^{2}(M \setminus U_{h^{1/2-\delta}})} = \O(h^{\infty}) \| \phi_h \|_{L^{2}(M)},
\end{equation*}
\end{itemize} such that the corresponding solution $u_h$ to \eqref{nls}
satisfies
\[
\| u_h \|_{L^\infty([0,T_h] ; L^{2}(M \setminus
  U_{h^{1/2-\delta}})) } \leq C h^{\nu},
\]
for $T_{h}=h^{p}$ in the case $p\in (0,4)\backslash \{1\}$, and $T_{h}=h^{1+\epsilon}$ 
in the case $p=1$.
\end{theorem}

\begin{rema}
In our current notations, the $\dH^1$ critical exponent for \eqref{nls} is
\[
p = \frac{4}{d-2} = \frac{4}{n-1},
\]
where $n$ is the dimension of the geodesic normal hypersurface to
$\Gamma$.  The $L^2$ critical exponent is $p = 4/d$, and the $\dH^{1/2}$ critical exponent is $p = 4/(d-1)$.

In general we construct highly concentrated
solutions along an elliptic orbit, which is effectively a $d-1$
dimensional soliton.  Since stable soliton solutions exist on
$\RR^d$ for monomial Schr\"odinger of the type in \eqref{nls} precisely
for $p < \frac{4}{d}$, this numerology matches well that required for the expected local
existence of stable nonlinear states on these lower dimensional
manifolds.

In addition, solving \eqref{nls} on compact manifolds with the power
$p = 4/(d-2)$ plays an important role in the celebrated Yamabe
problem, see \cite{SchYau}.  The authors hope that the
techniques here can give insight into the energy minimizers of such a problem
related to the geometry.
\end{rema}

\begin{rema}
In \cite{ChMa-Hyp}, the existence of nonlinear bound states on
hyperbolic space, $\mathbb{H}^d$, was explored.  In that case, it was
found that the geometry made compactness arguments at $\infty$ rather
simple, hence the only driving force for the existence of a nonlinear
bound state was the local behavior of the nonlinearity.  Work in progress with the second author,
third author, Michael Taylor and Jason Metcalfe is attempting to
generalize this observation to any rank one symmetric space.   In this note,
we find nonlinear quasimodes based purely on the local geometry, where
here the nonlinearity becomes a lower order correction, which is a
nicely symmetric result.  
\end{rema}

\begin{rema}
In a private conversation with Nicolas Burq, he has pointed out that
in settings where one assumes radial symmetry of the manifold, it is
possible to construct exact nonlinear bound states.
\end{rema}

\begin{Notations}
In this paper $c$, $C$ denote constants the value of which may change from line to line. We use the notations $a\sim b$,  
$a\lesssim b$ if $\frac1C b\leq a\leq Cb$, $a\leq Cb$ respectively. 
\end{Notations}

{\sc Acknowledgments.}  P. A. was supported by an NSF joint institutes postdoctoral
  fellowship (NSF grant DMS-0635607002) and a postdoctoral fellowship
  of the Foundation Sciences Math\'ematiques de Paris.  J.L.M. was
  supported in part by an NSF postdoctoral fellowship (NSF grant DMS-0703531) at Columbia
  University and a Hausdorff Center Postdoctoral Fellowship at the
  University of Bonn.  He would also like to thank both the Institut Henri
  Poincar\'e and the Courant Institute for being generous hosts during
part of the completion of this work.  H.C. was supported in part by
NSF grant DMS-0900524, and would also like to thank the University of
Bonn, Columbia University, the Courant Institute, and the Institut Henri Poincar\'e for their hospitality during
part of this work.  
  L.T. was supported in part by the
grants ANR-07-BLAN-0250 and ANR-10-JCJC 0109 of the Agence Nationale
de la Recherche.  In addition, the authors with to thank Rafe Mazzeo
and especially Colin Guillarmou for helpful conversations throughout
the preparation of this work.

\section{Outline of Proof}

The proof of Theorem \ref{T:1} is to solve approximately the
associated stationary equation.  That is, by separating variables in the $t$ direction, we
write
\[
\psi (x,t) = e^{-i \lambda t} u(x),
\]
from which we get the stationary equation
\[
(\lambda+ \Delta_g ) u = \sigma |u|^{p} u.
\]
For $h\ll1$ of the form \eqref{formeh}, the construction in the proof
finds a family of functions 
$$u_h(x) = h^{-(d-1)/4} g(h^{-1/2} x)$$ such that $g$ is rapidly decaying away from
$\Gamma$, $\Ci$, $g$ is normalized in $L^2$, and
\[
(\lambda+ \Delta_g ) u_h = \sigma |u_h|^{p} u_h +
Q(u_h),
\]
where $\lambda\sim h^{-2}$ and where the error $Q(u_h)$ is expressed by  a truncation of an
asymptotic series similar to that in \cite{Tho} and is of lower order in $h$.

The point is that Theorem \ref{T:1} (as well as Theorem \ref{T:2}) is an improvement over the trivial approximate
solution.  
It is well known that there exist quasimodes for the linear equation
localized near $\Gamma$ of the form
\[
v_h(x) = h^{-(d-1)/4} e^{is/h }f(s,h^{-1/2}x), \,\, (0<h\ll1),
\]
with $f$ a function rapidly decaying away from $\Gamma$, and $s$
a parametrization around $\Gamma$, so that
$v_h(x)$ satisfies
\[
(\lambda+\Delta_g )v_h = \O(h^{\infty}) \| v_h \|
\]
in any seminorm, see \cite{Ral-beam}.  Then
\[
(\lambda+ \Delta_g ) v_h= \sigma |v_h|^{p} v_h +
Q_2(v_h),
\]
where the error $Q_2(v_\lambda) = | v_h|^p v_h$ satisfies
\[
\| Q_2(v_h) \|_{\dH^s} = \O(h^{-s - p(d-1)/4} ).
\]


\section{A toy model}
\label{S:toy}

In this section we consider a toy model in two dimensions, and we give an idea of the proof of Theorem \ref{T:2}. For simplicity, we moreover assume that $2\leq p<4$. As it is a
toy model, we will not dwell on error analysis, and instead make
Taylor approximations at will without remarking on the error terms.
Consider the manifold
\[
M = \reals_x / 2 \pi \ZZ \times \reals_\theta / 2 \pi
\ZZ, 
\]
equipped with a metric of the form
\[
ds^2 = d x^2 + A^2(x) d \theta^2,
\]
where $A \in \Ci$ is a smooth function, $A \geq \epsilon>0$ for some
$\epsilon$.  
From this metric, we get the volume form
\[
d \Vol = A(x) dx d \theta,
\]
and the Laplace-Beltrami operator acting on $0$-forms
\[
\Delta_g f = (\partial_x^2 + A^{-2} \partial_\theta^2 + A^{-1}
A' \partial_x) f.
\]
We observe that we can conjugate
$\Delta_g$ by an isometry of metric spaces and separate variables so
that spectral analysis of $\Delta_g$ is equivalent to a one-variable
semiclassical problem with potential.  That is, let $S : L^2(X, d
\Vol) \to L^2(X, dx d \theta)$ be the isometry given by
\[
Su(x, \theta) = A^{1/2}(x) u(x, \theta).
\]
Then $\tDelta = S \Delta S^{-1}$ is essentially self-adjoint on $L^2 (
X, dx d \theta)$ with mild assumptions on $A$.  A simple calculation gives
\[
-\tDelta f = (- \partial_x^2 - A^{-2}(x) \partial_\theta^2 + V_1(x) )
f,
\]
where the potential
\[
V_1(x) = \frac{1}{2} A'' A^{-1} - \frac{1}{4} (A')^2 A^{-2}.
\]

We are interested in the nonlinear Schr\"odinger equation \eqref{nls},
so we make a separated Ansatz:
\[
u_\lambda(t, x, \theta) = e^{-i t \lambda} e^{i k \theta} \psi(x),
\]
where $k \in \ZZ$ and $\psi$ is to be determined (depending on both
$\lambda$ and $k$).  
Applying the Schr\"odinger operator (with
$\tDelta$ replacing $\Delta$) to
$u_\lambda$ yields the equation
\[
(D_t + \tDelta) e^{i t \lambda} e^{i k \theta} \psi(x) = (\lambda +
\partial_x^2 - k^2 A^{-2}(x) + V_1(x) ) e^{i t \lambda} e^{i k \theta}
\psi(x) = \sigma | \psi|^p e^{i t \lambda} e^{i k \theta} \psi(x),
\]
where we have used the standard notation $D =- i \p$.
We are interested in the behaviour of a solution or approximate
solution near an elliptic periodic geodesic, which occurs at a maximum
of the function $A$.  For simplicity, let
\[
A(x) = \sqrt{ (1 + \cos^2(x))/2 },
\]
so that in a neighbourhood of $x = 0$, $A^2 \sim 1 - x^2$ and $A^{-2}
\sim 1 + x^2$.  The function $V_1(x) \sim \text{const.}$ in a
neighbourhood of $x = 0$, so we will neglect $V_1$.  If we assume
$\psi(x)$ is localized near $x=0$, we get the stationary reduced equation
\[
(-\lambda + \partial_x^2 - k^2 ( 1 + x^2 ) ) \psi =- \sigma | \psi |^p
\psi .
\]
Let $h = |k|^{-1}$ and use the rescaling operator $T \psi(x) = T_{h,0}
\psi(x) = h^{-1/4} \psi(h^{-1/2} x )$ (see Lemma
\ref{L:T-lemma} below with $n=1$) to conjugate:
\[
T^{-1} (-\lambda + \partial_x^2 - k^2 ( 1 + x^2 ) ) T T^{-1} \psi =
T^{-1} (\sigma | \psi |^p
\psi )
\]
or
\[
( -\lambda + h^{-1} \partial_x^2 - k^2(1 + h x^2 ) \phi = \sigma h^{-p/4} |
\phi |^p \phi,
\]
where $\phi = T^{-1} \psi$.  Let us now multiply by $h$:
\[
( - \partial_x^2 + x^2 - E ) \phi = \sigma h^{q} | \phi |^p \phi,
\]
where
\[
E := \frac{1 - \lambda h^2}{h}
\]
and
\[
q: = 1 - p \frac{d-1}{4} = 1 - \frac{p}{4}.
\]
Observe the range restriction on $p$ is precisely so that
\[
0 < q \leq 1/2.
\]
We make a WKB type Ansatz, although in practice we will only take
two terms (more is possible if the nonlinearity is smooth as in the context of Theorems \ref{T:0} and \ref{T:1}):
\[
\phi =\phi_0 + h^q \phi_1, \,\,\, E =E_0 + h^q E_1.
\]
The first two equations are
\begin{align*}
& h^0: \quad ( - \partial_x^2 + x^2 - E_0 ) \phi_0 =  0, \\
& h^{q}: \quad ( - \partial_x^2 + x^2 - E_0 -h^q E_1) \phi_1 = E_1 \phi_0 +
\sigma | \phi_0 |^p \phi_0.
\end{align*}
Observe we have included the $h^q E_1 \phi_1$ term on the left hand
side.

The first equation is easy:
\[
\phi_0(x) = e^{-x^2/2}, \,\,\, E_0 = 1.
\]
For the second equation, we want to project away from $\phi_0$ which
is in the kernel of the operator on the left hand side.  That is,
choose $E_1$ satisfying
\[
\lll  E_1 \phi_0 + \sigma | \phi_0 |^p \phi_0 , \phi_0 \rrr = 0,
\]
so that the right hand side is in $\phi_0^\perp \subset L^2$.  
Then since the spectrum of the one-dimensional harmonic oscillator is
simple (and of the form $(2m+1)$, $m \in \ZZ$), the operator $(
- \partial_x^2 + x^2 - E_0 -h^q E_1)$ is invertible on
$\phi_0^\perp \subset L^2 $ with inverse bounded by $(2 - h^q)^{-1}$.
Hence for $h>0$ sufficiently small, we can find
$\phi_1 \in L^2$ satisfying the second equation above (here we have
used that $\phi_0$ is Schwartz with bounded $H^s$ norms).  Further,
since $\phi_0$ is Schwartz and strictly positive, so is $|\phi_0 |^p \phi_0$, so by
propagation of singularities, $\phi_1$ is also Schwartz.  In particular, both
$\phi_0$ and $\phi_1$ are rapidly decaying away from $x = 0$.

Let now $\psi(x) = T(\phi_0(x) + h^{q} \phi_1(x) )$, and observe
that by the above considerations, $\psi(x)$ is $\O(h^\infty)$ in any
seminorm, outside an $h^{1/2 - \delta}$ neighbourhood of $x = 0$.  Let
$u = e^{i t \lambda} e^{i k \theta} \psi(x)$ so that $\| u \|_{L^2(dx
  d \theta)} \sim 1$, and $u$ is $\O(h^\infty)$ outside an $h^{1/2 -
  \delta}$ neighbourhood of $x = 0$.  Furthermore, $u$ satisfies the
equation (again neglecting smaller terms)
\begin{align*}
(D_t + \tDelta) u & =  h^{-1} T  h T^{-1} (-\lambda +
\partial_x^2 - k^2 A^{-2}(x) ) T T^{-1} u\\
 & = h^{-1} T (\sigma h^{q} | T^{-1} u|^p T^{-1} u +  \O(h |
 T^{-1} u |^{p+1}) \\
& = \sigma | u |^{p} u + Q,
\end{align*}
where $Q$ satisfies the pointwise bound
\[
Q = \O( h^{q} | u |^{p+1}).
\]

We now let $\tilde{u}$ be the actual solution to \eqref{nls} with the
same initial profile:
\[
\begin{cases}
(D_t + \Delta ) \tilde{u} = \sigma | \tilde{u} |^p \tilde{u} \\
\tilde{u} |_{t = 0 } = e^{ik \theta} \psi(x).
\end{cases}
\]
Set $T_{h}=h^{p}$, then with the Strichartz estimates of Burq-G\'erard-Tzvetkov \cite{BGT-comp}  we prove (see Proposition \ref{properreur2}) that there exists $\nu>0$ so that 
\begin{equation*}
\| u - \tilde{u} \|_{L^\infty([0,T_h];L^{2}(M))} \leq Ch^{\nu},
\end{equation*} 
 and therefore we can compute:
\begin{align*}
\| \tilde{u} \|_{L^\infty([0,T_h]; L^2( M \setminus U_{h^{1/2-\delta
      }}))} & \leq \| u \|_{L^\infty([0,T]; L^2( M \setminus U_{h^{1/2-\delta}}))}  + \| \tilde{u} - u \|_{L^\infty([0,T]; L^2( M \setminus U_{h^{1/2-\delta}}))} \\
& = \O(h^{\infty}) + \O(h^{\nu})=\O(h^{\nu}),
\end{align*}
 which gives the result.

\begin{rema}
It is very important to point out that the sources of additional error
in this heuristic exposition have been ignored, and indeed, to apply a
similar idea in the general case, a microlocal reduction to a {\it
  tubular neighbourhood} of $\Gamma$ in cotangent space is employed.
The function $\phi_0$ is no longer so simple, and the nonlinearity $|
\phi_0 |^p \phi_0$ is no longer necessarily smooth.  Because of this,
the semiclassical wavefront sets are no longer necessarily compact, so
a cutoff in frequency results in a {\it fixed loss}.  
\end{rema}

\begin{rema}
We remark that the Strichartz estimates from \cite{BGT-comp} are sharp
on the  sphere $\SS^d$ for a particular Strichartz pair, but this
is not necessarily true on a generic Riemannian manifold.  See
\cite{BSS-Sch} for a thorough discussion of this fact.
\end{rema}


\section{Preliminaries}


\label{preliminaries}

\subsection{Symbol calculus on manifolds}
This section contains some basic definitions and results from semiclassical and microlocal 
analysis which we will be using throughout the paper.  This is essentially standard, but we 
include it for completeness.  The techniques presented have been
established in multiple references, including but not limited to the
previous works of the second author \cite{Chr-NC,Chr-QMNC},
Evans-Zworski \cite{EvZw}, Guillemin-Sternberg
\cite{GuSt-Geom,GuSt-Semi}, H\"ormander \cite{Hor-v1,Hor-v2},
Sj\"ostrand-Zworski \cite{SjZw-mono}, Taylor \cite{Tay-pdo},
and many more.

To begin we present results from \cite{EvZw}, Chapter $8$ and Appendix
$E$.  Let $X$ be a smooth, compact manifold.  We will be operating on half-densities,
\begin{eqnarray*}
u(x)|dx|^{\frac{1}{2}} \in \Ci\left(X, \Omega_X^{\frac{1}{2}}\right),
\end{eqnarray*}
with the informal change of variables formula
\begin{eqnarray*}
u(x)|dx|^{\frac{1}{2}} = v(y)|dy|^\half, \,\, \text{for}\,\, y = \kappa(x) 
\Leftrightarrow v(\kappa(x))|\kappa'(x)|^\half = u(x),
\end{eqnarray*}
where $|\kappa'(x)|$ has the canonical interpretation as the Jacobian $| \det (\p
\kappa) |$.  
By symbols on $X$ we mean the set
\begin{eqnarray*}
\lefteqn{\s^{k,m} \left(T^*X, \Omega_{T^*X}^\half \right):= } \\
& = &\left\{ a \in \Ci(T^*X \times (0,1], \Omega_{T^*X}^\half): \left| \partial_x^\alpha 
\partial_\xi^\beta a(x, \xi; h) \right| \leq C_{\alpha \beta}h^{-m} \langle \xi \rangle^{k - |\beta|} \right\}.
\end{eqnarray*}
Essentially this is interpreted as saying that on each coordinate
chart, $U_\alpha$, of $X$, $a \equiv a_\alpha$ where is $a_\alpha \in
\s^{k,m}$ in a standard symbol on $\reals^d$.
There is a corresponding class of pseudodifferential operators $\Psi_h^{k,m}(X, \Omega_X^\half)$ 
acting on half-densities defined by the local formula (Weyl calculus)
in $\reals^n$:\footnote{We use the semiclassical, or rescaled {\it
    unitary} Fourier transform throughout: \[\mathcal{F}_h u( \xi) =
  (2 \pi h)^{-d/2} \int e^{-i \langle x, \xi \rangle/h} u(x) d x.\]}
\begin{eqnarray*}
\Op_h^w(a)u(x) = \frac{1}{(2 \pi h)^n} \int \int a \left( \frac{x + y}{2}, \xi; h \right) 
e^{i \langle x-y, \xi \rangle / h }u(y) dy d\xi.
\end{eqnarray*}
We will occasionally use the shorthand notations $a^w := \Op_h^w(a)$ and $A:=\Op_h^w(a)$ when 
there is no ambiguity in doing so. 

We have the principal symbol map
\begin{eqnarray*}
\sigma_h : \Psi_h^{k,m} \left( X, \Omega_X^\half \right) \to \s^{k,m} \left/ \s^{k-1, m-1} 
\left(T^*X, \Omega_{T^*X}^\half \right) \right.,
\end{eqnarray*}
which gives the left inverse of $\Op_h^w$ in the sense that 
\begin{eqnarray*}
\sigma_h \circ \Op_h^w: \s^{k,m} \to \s^{k,m}/\s^{k-1, m-1} 
\end{eqnarray*}
is the natural projection.  Acting on half-densities in the Weyl calculus, the principal 
symbol is actually well-defined in $\s^{k,m} / \s^{k-2,m-2}$, that is, up to $\O(h^2)$ in 
$h$ (see, for example \cite{EvZw}, Appendix E).  

We will use the notion of semiclassical wave front sets for
pseudodifferential operators on manifolds, see H\"ormander
\cite{Hor-v2}, \cite{GuSt-Geom}.  
If $a \in \s^{k,m}(T^*X, \Omega_{T^*X}^\half)$, we define the singular support or essential support for $a$:
\begin{eqnarray*}
\esssupp_h a \subset T^*X \bigsqcup \SS^*X,
\end{eqnarray*}
where $\SS^*X = (T^*X \setminus \{0\}) / \reals_+$ is the cosphere bundle (quotient taken with 
respect to the usual multiplication in the fibers), and the union is
disjoint.  The $\esssupp_h a$ is defined using complements:
\begin{eqnarray*}
\lefteqn{\esssupp_h a := } \\
&& \complement \left\{ (x, \xi) \in T^*X : \exists \epsilon >0, \,\,\, (\partial_x^\alpha 
\partial_\xi^\beta a) (x', \xi') = \O(h^\infty), \,\,\, d(x, x') + |\xi - \xi'| < \epsilon \right\} \\
&& \bigcup \complement \{ (x, \xi) \in T^*X \setminus 0 : \exists \epsilon > 0, \,\,\, (\partial_x^\alpha 
\partial_\xi^\beta a) (x', \xi') = \O (h^\infty \langle \xi \rangle^{-\infty}),  \\
&& \quad \quad \quad d(x, x') + 1 / |\xi'| + | \xi/ |\xi| - \xi' /
|\xi'| | < \epsilon \} / \reals_+.
\end{eqnarray*}
We then define the wave front set of a pseudodifferential operator $A \in \Psi_h^{k,m}( X, \Omega_X^\half )$:
\begin{eqnarray*}
\WF(A) : = \esssupp_h(a), \,\,\, \text{for} \,\,\, A = \Op_h^w(a).
\end{eqnarray*}

If $u(h)$ is a family of distributional half-densities, $u \in \Ci( (0, 1]_h, \mathcal{D}'(X,
\Omega_X^\half))$, we say $u(h)$ is $h$-{\it tempered} if there is an $N_0$ so that $h^{N_0}u$ is bounded in $\mathcal{D}'(X, \Omega_X^\half)$.
If $u = u(h)$ is an $h$-tempered family of distributions, we can 
define the semiclassical wave front set of $u$, again by complement:
\begin{eqnarray*}
\lefteqn{\WF (u) := } \\
&& \complement \{(x, \xi) : \exists A \in \Psi_h^{0,0}, \,\, \text{with} \,\, \sigma_h(A)(x,\xi) \neq 0, \,\,  \\ 
&& \quad \text{and} \,\,Au \in h^\infty \Ci((0,1]_h, \Ci(X, \Omega_X^\half)) \}.
\end{eqnarray*}
For $A = \Op_h^w(a)$ and $B = \Op_h^w(b)$, $a \in \s^{k,m}$, $b \in \s^{k',m'}$ we have the composition 
formula (see, for example, \cite{DiSj})
\begin{eqnarray}
\label{Weyl-comp}
A \circ B = \Op_h^w \left( a \# b \right),
\end{eqnarray}
where
\begin{eqnarray}
\label{a-pound-b}
\s^{k + k', m+m'} \ni a \# b (x, \xi) := \left. e^{\frac{ih}{2} \omega(Dx, D_\xi; D_y, D_\eta)} 
\left( a(x, \xi) b(y, \eta) \right) \right|_{{x = y} \atop {\xi = \eta}} ,
\end{eqnarray}
with $\omega$ the standard symplectic form. 

We record some useful Lemmas.

\begin{lemma}
Suppose $(x_0, \xi_0) \notin \WF (u)$.  Then $\forall b \in \Ci_c(T^*
\reals^n)$ with support sufficiently close to $(x_0, \xi_0)$ we have
\[
b(x, hD) u = \O_{\mathcal{S}} ( h^\infty ).
\]
\end{lemma}
Here $ \O_{\mathcal{S}} ( h^\infty )$ means $\O(h^\infty)$ in any
Schwartz semi-norm.  The proof of this Lemma follows similarly to that
of Theorem $8.9$ in \cite{EvZw}.

\begin{theorem}
\label{T:wf-properties}
\begin{itemize}

\item[(i)] Suppose $a \in \s(m)$ and $u(h)$ is $h$-tempered.  Then
\[
\WF (a^w u) \subset \WF (u) \cap \esssupp_h (a).
\]

\item[(ii)] If $a \in \s(m)$ is real-valued, then also
\[
\WF ( u ) \subset \WF (a^w u) \cup \overline{\complement \{\esssupp_h a  \}}.
\]

\end{itemize}

\end{theorem}

\begin{proof}

Assertion $1$ is straightforward.  The proof of assertion $2$ is
standard, however we present it here so we can use it for the
analogous result for the blown-up wavefront set.

We will show if $a(x_0, \xi_0) \neq 0 $ and $a^w u =
\O_{L^2}(h^\infty)$ then there exists $b$, $b(x_0, \xi_0) \neq 0$ so
that $b^w u = \O_{L^2}(h^\infty)$.  There exists a neighbourhood $U
\ni (x_0, \xi_0)$, a real-valued function $\chi$, and a positive
number $\gamma >0$ such that $\supp \chi \cap U = \emptyset$ and 
\[
| a + i \chi| \geq \gamma \text{ everywhere}.
\]
Then $P = a^w + i \chi^w$ has an approximate left inverse $c^w$ so
that
\[
c^w P = \id + R^w,
\]
where $R^w = \O_{L^2 \to L^2}(h^\infty)$.  Choose $b \in \s$ so that
$\supp (b) \subset U$ and $b(x_0, \xi_0) \neq 0$.  Then $b^w \chi^w =
\O(h^\infty)$ as an operator on $L^2$.  Hence
\begin{align*}
b^w u = &  b^w c^w P u - b^w R^w u \\
= & b^w c^w a^w u + i b^w c^w \chi^w u - b^w R^w u \\
= & \O(h^\infty),
\end{align*}
where we have used the Weyl composition formula to conclude
$\esssupp_h( c \# \chi) \cap \supp b = \emptyset$.
\end{proof}

\subsection{Exotic symbol calculi}

Following ideas from \cite{SjZw-mono}, since rescaling often means dealing with symbols with bad decay
properties, we introduce weighted wave front sets as well.  Let us
first recall the non-classical symbol classes: 
\begin{eqnarray*}
\lefteqn{\s^{k,m}_{\delta, \gamma} \left(T^*X, \Omega_{T^*X}^\half \right):= } \\
& = &\left\{ a \in \Ci(T^*X \times (0,1], \Omega_{T^*X}^\half): \left| \partial_x^\alpha 
\partial_\xi^\beta a(x, \xi; h) \right| \leq C_{\alpha
\beta}h^{-\delta|\alpha| -\gamma | \beta|-m} \langle \xi \rangle^{k - |\beta|} \right\}.
\end{eqnarray*}
That is, symbols which lose $\delta$ powers of $h$ upon
differentiation in $x$ and $\gamma$ powers of $h$ upon differentiation
in $\xi$.  Note the simplest way to achieve this is to take a
symbol $a(x, \xi) \in \s$ and rescale $(x, \xi) \mapsto (h^{-\delta} x, h
^{-\gamma} \xi)$, which then localizes on a scale $h^{\delta + \gamma}$ in
phase space.  We thus make the restriction that $0 \leq \delta, \gamma \leq
1$, $0 \leq \delta + \gamma \leq 1$, and to gain powers of $h$ by integrations by parts, we usually
also require $\delta + \gamma < 1$.  We can define wavefront sets using
$\s_{\delta, \gamma} = \s^{0,0}_{\delta,\gamma}$ symbols, but the localization of the wavefront sets is stronger.

\begin{definition}

 If $u(h)$ is $h$-tempered and $0 \leq \delta , \gamma \leq 1$, $0
 \leq \delta + \gamma \leq 1$, 
\begin{align*}
\dWF(u) = & \complement \{ (x_0 ,\xi_0) : \exists a \in
\s_{\delta,\gamma} \cap \Ci_c, \,  a(x_0, \xi_0)
\neq 0, a(x,\xi) = \tilde{a}(h^{-\delta}x, h^{-\gamma}\xi) \\
& \text{ for some } \tilde{a}
\in \s \text{ and }\, a^w u = \O_{L^2}(h^\infty) \}.
\end{align*}

\end{definition}

We have the following immediate corollary.

\begin{corollary}
\label{C:dWF-cor}

If $a \in \s_{\delta, \gamma}(m)$, $0 \leq \delta + \gamma< 1$, and $a$ is real-valued, then 
\[
\dWF ( u ) \subset \dWF (a^w u) \cup \overline{\complement \{ \esssupp_h a  \}}.
\]

\end{corollary}

The proof is exactly the same as in Theorem \ref{T:wf-properties} only
all symbols must scale the same, so they must be in $\s_{\delta,
  \gamma}$.  
In
order to conclude the existence of approximate inverses, we need the
restriction $\delta + \gamma < 1$, and the rescaling
operators from \S \ref{S:prelim-rescaling} which can be used to reduce to the familiar $h^{-\delta'}$
calculus, where $\delta' = (\delta + \gamma)/2$, by replacing $h$ with
$h^{\delta -\gamma}$.

\subsection{Rescaling operators}
\label{S:prelim-rescaling}

We would like to introduce $h$-dependent rescaling operators.  The rescaling operators should be unitary
with respect to natural Schr\"odinger energy norms, namely the
homogeneous $\dH^s$ spaces.  Let us recall in $\reals^n$, the $\dH^s$
space is defined as the completion of $\s$ with respect to the
topology induced by the inner product
\[
\lll u , v \rrr_{\dH^s} = \int_{\reals^n} | \xi |^{2s} \hat{u}(\xi)
\overline{\hat{v}}(\xi) d \xi,
\]
where as usual $\hat{u}$ denotes the Fourier transform.  The $\dH^0$
norm is just the $L^2$ norm, and the $\dH^1$ norm is 
\[
\| u \|_{\dH^1} \simeq \| |\nabla u | \|_{L^2}.
\]
The purpose in taking the homogeneous norms instead of the usual
Sobolev norms is to make the rescaling operators in the next Lemma
unitary.

\begin{lemma}
\label{L:T-lemma}
For any $s \in \reals$, $h>0$, the linear operator $T_{h,s}$ defined by
\[
T_{h,s} w(x) = h^{s/2-n/4}w(h^{-1/2} x)
\]
is unitary on $\dH^s( \reals^n)$, and for any $r \in \reals$, 
\begin{align*}
\| T_{h,s} w \|_{\dH^r} & = h^{(s-r)/2} \| w \|_{\dH^r}, \\
\| T_{h,0} w \|_{\dH^s} & = h^{-s/2} \| w \|_{\dH^s}.
\end{align*}

Moreover, for any pseudodifferential
operator $P(x, D)$ in the Weyl calculus,
\[
T_{h,s}^{-1} P(x,D) T_{h,s} = P(h^{1/2} x, h^{-1/2} D).
\]
\end{lemma}

\begin{rema}
Observe that for this lemma, the usual assumption that $h$ be small is not
necessary.  
\end{rema}

\begin{proof}
The proof is simple rescaling, but we include it here for the
convenience of the reader.  To check unitarity, we just change variables:
\begin{align*}
\lll u, T_{h,s} w \rrr_{\dH^s} & = \lll | \xi |^s \hat{u}, | \xi |^s
\widehat{ T_{h,s} w } \rrr \\
& = h^{s/2+n/4} \int | \xi |^{2s} \hat{u}( \xi )
\overline{\hat{w}(h^{1/2} \xi)} d \xi \\
& = h^{-n/4-s/2} \int | \xi |^{2s} \hat{u}( h^{-1/2} \xi )
\overline{\hat{w}(\xi)} d \xi \\
& = h^{n/4 -s/2} \int | \xi |^{2s} \widehat{ u(h^{1/2}\cdot)}(\xi)
\overline{\hat{w}(\xi)} d \xi \\
& = \lll T_{h,s}^{-1} u, w \rrr_{\dH^s}.
\end{align*}

To check the conjugation property, we again compute
\begin{align*}
T_{h,s}^{-1} P(x,D) T_{h,s} w(x) & = T_{h,s}^{-1} (2 \pi )^{-n}
h^{s/2-n/4} \int
e^{i\lll x-y , \xi \rrr} P(\frac{x+y}{2}, \xi) w(h^{-1/2} y) dy d \xi
\\
& = T_{h,s}^{-1} (2 \pi )^{-n}
h^{s/2+n/4} \int e^{i \lll x-h^{1/2} y, \xi \rrr } P(\frac{x + h^{1/2}
  y}{2}, \xi ) w(y) d y d \xi \\
& = T_{h,s}^{-1} (2 \pi )^{-n}
h^{s/2-n/4} \int e^{i \lll x-h^{1/2} y, h^{-1/2} \xi \rrr } P(\frac{x + h^{1/2}
  y}{2}, h^{-1/2} \xi )w(y) d y d \xi \\
& = (2 \pi )^{-n}\int e^{i \lll x- y,  \xi \rrr } P(\frac{h^{1/2}  ( x + 
  y)}{2}, h^{-1/2} \xi ) w(y) d y d \xi \\
& = P(h^{1/2} x, h^{-1/2} D) w(x).
\end{align*}

\end{proof}

The purpose of using the rescaling operators $T_{h,s}$ is that if $u
\in \dH^s$ has $h$-wavefront set
\[
\dWF (u) \subset \{ | x | \leq \alpha(h) \text{ and } | \xi | \leq
\beta(h) \},
\]
where, according to the uncertainty principle, $\alpha(h) \beta(h)
\geq h$, then $T_{h,s} u$ has $h$-wavefront set
\[
\sWF_{h, \delta-1/2, \gamma + 1/2} (T_{h,s} u) \subset \{ | x | \leq \alpha(h) h^{1/2} \text{ and } | \xi | \leq
\beta(h) h^{-1/2} \},
\]
provided, of course, that $\delta \geq 1/2$.  
To see this, we just observe that for any $\psi \in
\Ci_c(\reals^{2n})$, $\psi \equiv 0$ on $\{ | x | \leq 1, | \xi | \leq
1 \}$, we have
\[
\psi( x/\alpha(h) , D / \beta(h) ) u = \O(h^\infty ) \| u \|_{L^2},
\]
or any other semi-norm, and hence
\begin{align*}
\psi( h^{-1/2} x/\alpha(h) , h^{1/2} D / \beta(h) ) T_{h,s} u & = 
T_{h,s} \psi( x/\alpha(h) , D / \beta(h) ) T_{h,s}^{-1} T_{h,s} u \\
& = T_{h,s} \psi( x/\alpha(h) , D / \beta(h) ) u \\
& = h^{s/2} \O(h^\infty) \| u \|_{L^2} \\
& = \O(h^\infty) \| u \|_{L^2}.
\end{align*}
Finally, we note that the symbol of the operator on the left-hand side
is zero on the set  
\[
\{ | x | \leq h^{1/2} \alpha(h), \text{ and } | \xi | \leq h^{-1/2}
\beta(h) \},
\]
and any such symbol in $\s_{\delta - 1/2, \gamma + 1/2}$, elliptic at a
point outside this set, can be locally 
obtained in this fashion.

\section{Geometry near an elliptic geodesic}

Suppose $\Gamma$ is a geodesic in a $n+1$ dimensional Riemannian manifold.
Following \cite[\S 2.1]{Mazzeo-Pacard} (cf. \cite{Tubes}) we fix an arclength parametrization $\gamma(t)$ of $\Gamma$, and a parallel orthonormal frame $E_1, \ldots, E_n$ for the normal bundle $N\Gamma$ to $\Gamma$ in $M$. This determines a coordinate system
\begin{equation*}
	x= (x_0, x_1, \ldots, x_n) \mapsto \exp_{\gamma(x_0)}(x_1 E_1 + \ldots x_n E_n) = F(x).
\end{equation*}
We write $x' = (x_1, \ldots, x_n)$ and use indices $j,k, \ell \in \{ 1, \ldots, n\}$, $\alpha, \beta, \delta \in \{ 0, \ldots, n \}$.
We also use $X_\alpha = F_*(\partial_{x_{\alpha}})$.

Note that $r(x) = \sqrt{x_1^2 + \ldots x_n^2}$ is the geodesic distance from $x$ to $\Gamma$ and $\partial_r$ is the unit normal to the geodesic tubes $\{ x: d(x, \Gamma)= \text{cst} \}$.
Let $p = F(x_0, 0)$, $q = F(x_0, x')$, and $r = r(q) = d(p,q)$ then we
have \cite[Proposition 2.1]{Mazzeo-Pacard}, which states
\begin{equation}\label{LocalMet}
\begin{split}
	g_{jk}(q) &= \delta_{ij} + \frac13 g(R(X_s, X_j) X_\ell, X_k
        )_p x_s x_\ell + \mathcal{O}(r^3) , \\
	g_{0k}(q) &= \mathcal{O}(r^2) , \\
	g_{00}(q) &= 1 - g(R(X_k, X_0) X_0, X_{\ell} )_p x_k x_\ell +
        \mathcal{O}(r^3) , \\
	\Gamma_{\alpha\beta}^{\delta} &= \mathcal{O}(r) , \\
	\Gamma_{00}^{k} &= -\sum_{j=1}^n g(R(X_k, X_0)X_j, X_0)_p x_j
        + \mathcal{O}(r^2) ,
\end{split}
\end{equation}
where
$\Gamma^\delta_{\alpha\beta} = \frac{1}{2} g^{\delta\eta} ( X_\alpha g_{\eta\beta} + X_{\beta} g_{\alpha\eta} - X_\eta g_{\alpha\beta}).$

In these coordinates, the Laplacian has the form
\begin{equation}\label{LocalLap}
\begin{split}
	\Delta_g & :=  \frac{1}{\sqrt{\det g}} \text{div} (\sqrt{\det g} g^{-1} \nabla) \\
	&= g^{jk} X_j X_k - g^{jk}\Gamma_{jk}^\ell X_\ell \\
	&= g^{00} X_0X_0 + 2 g^{k0} X_kX_0 - g^{jk}\Gamma^0_{jk} X_0 
	+2 g^{0j}\Gamma_{0j}^k X_k + \Delta_{\Gamma^{\perp}} ,
\end{split}
\end{equation}
where $\Delta_{\Gamma^{\perp}}$ is the Laplacian in the directions transverse to $\Gamma$.

Denote the geodesic flow on $SM$, the unit tangent bundle, by $\phi_t$.
Let $\gamma(0) = p \in \Gamma$ and $\zeta = \gamma'(0) \in SM$.
Associated to $\Gamma$ is a periodic orbit $\phi_t \zeta$ of the geodesic flow on $SM$. 
This orbit $\phi_t \zeta$ is called {\em stable} if, whenever $\mathcal{V}$ is a tubular neighbourhood of $\phi_t \zeta$, there is a neighbourhood $\mathcal{U}$ of $\zeta$ such that $\zeta' \in \mathcal{U}$ implies $\phi_t\zeta' \subseteq \mathcal{V}$.
Given a hypersurface $\Sigma$ in $SM$ containing $\zeta$ and transverse to $\phi_t \zeta$, we can define a {\em Poincar\'e map} $\mathcal{P}$ near $\zeta$, by assigning to each $\zeta'$ the next point on $\phi_t(\zeta')$ that lies in $\Sigma$.
Two Poincar\'e maps of $\phi_t\zeta$ are locally conjugate and hence the eigenvalues of $d\mathcal{P}$ at $\zeta$ are invariants of the periodic orbit $\phi_t\zeta$.

The Levi-Civita connection allows us to identify $TTM$ with the sum of a horizontal space and a vertical space, each of which can be identified with $TM$.
Thus we can choose $\Sigma$ so that $T_\zeta \Sigma$ is equal to $E \oplus E$, where $E$ is the orthogonal complement of $\zeta$ in $T_pM$.
The linearized Poincar\'e map is then given by 
\begin{equation*}
	\xymatrix @R=1pt 
	{ E \oplus E \ar[r]^{d\mathcal{P}} & E \oplus E , \\ 
	(V, W) \ar@{|->}[r] & (J( \mathrm{length}(\gamma) ), \nabla_{X_0} J( \mathrm{length}(\gamma) ) ) },
\end{equation*}
where $J$ is the unique Jacobi field along $\gamma$ with $J(0) = V$ and $\nabla_{X_0} J(0) = W$, i.e., $J$ solves
\begin{equation*}
	\begin{cases}
		\nabla_{X_0} \nabla_{X_0} J + R(J, X_0) X_0 = 0 \\
		J(0) = V, \quad \nabla_{X_0} J(0) = W.
	\end{cases}
\end{equation*}
The linearized map $d\mathcal{P} = \left.d\mathcal{P}\right|_{\zeta}$ preserves the symplectic form on $E \oplus E$,
\begin{equation*}
	\omega( (V_1, W_1), (V_2, W_2) ) = \langle V_1, W_2 \rangle - \langle W_1, V_2 \rangle, 
\end{equation*}
and so its eigenvalues come in complex conjugate pairs. 
We say that $\Gamma$ is a {\em non-degenerate elliptic closed geodesic} if the eigenvalues of $d\mathcal{P}$ have the form $\{ e^{\pm i \alpha_j} : j = 1, \ldots, n \}$ where each $\alpha_j$ is real and the set $\{ \alpha_1, \ldots, \alpha_n, \pi \}$ is linearly independent over $\mathbb{Q}$.

From \eqref{LocalMet}, the Hessian of the function $g_{00}(q)$ as a function of $x'$ is (minus) the transformation appearing in the Jacobi operator
\begin{equation*}
	E \ni V \mapsto R(V, X_0) X_0 \in E.
\end{equation*}
Notice that if $V$ is a normalized eigenvector for this operator, with eigenvalue $\lambda$, then
\begin{equation*}
	\mathrm{sec}_p(X_0, V) = g_p(R(V, X_0)X_0, V) = \lambda g_p(V, V) = \lambda.
\end{equation*}
The very useful property
\begin{equation}\label{Assumption}
	{ p \in \Gamma, V \in E \implies \mathrm{sec}_p(X_0, V) > 0 },
\end{equation}
holds for any elliptic closed geodesic.

\begin{rema}
One can verify \eqref{Assumption} by means of the Birkhoff normal form (see \cite[\S10.3]{Zel-survey}).
Indeed, if one of the sectional curvatures were negative, then the Birkhoff normal form of the linearized Poincare map (in $T^*M$) must have an eigenvalue off the unit circle.  If so, then there is at least one nearby orbit which does not stay nearby, hence the periodic geodesic is not elliptic.  Similarly, if one of the curvatures vanishes, then the linearized Poincare map has a zero eigenvalue, and hence the logs of the eigenvalues are not independent from $\pi$ over the rationals (in other words, the Poincare map is degenerate, and not even symplectic).  
\end{rema}

\section{Compact Solitons: The nonlinear Ansatz}

We are interested in finding quasimodes for the non-linear Schr\"odinger equation
\begin{eqnarray*}
	-\Delta_g u = \lambda u - \sigma | u |^{p} u,
\end{eqnarray*}
where $\sigma = \pm 1$ determines if we are in the focussing or
defocussing case.  We will construct $u$ approximately solving this
equation with $u$ concentrated near $\Gamma$ in a sense to be made
precise below.  

We take as Ansatz 
\begin{equation*}
	F(x, h) = e^{ix_0/h} f(x, h), \,\,\, h^{-1} = \frac{2 \pi N}{L},
\end{equation*}
where $L>0$ is the period of $\Gamma$, 
and assume for the time being that the function $f$ is concentrated in
an $h$-dependent neighbourhood of $\Gamma$.  We are going to employ
a semiclassical reduction, and we are interested in fast oscillations
($h \to 0$), so we assume $\sWF_{h,1/2-\delta, 0} f(x,h) \subset \{ | x'| \leq \epsilon
h^{1/2 - \delta }, | \xi' | \leq \epsilon \}$ for
some $\epsilon, \delta >0$\footnote{The reason for the weaker
  concentration in frequency $| \xi'|$ is that the nonlinearity forces
working with non-smooth functions, so some decay at infinity in
frequency is lost.}.  The localization property of $f$ as constructed
later will be verified
in Section \ref{S:quasimodes}.  
We compute from \eqref{LocalLap}, with $\Delta$ the non-positive Laplacian,
\begin{multline}
\label{E:DeltaF}
	\Delta F
	= e^{ix_0/h} \left[
	g^{00} \left( -\frac1{h^2} f + \frac{2i}h X_0 f + X_0X_0 f \right)
	+ 2 g^{k0}  X_k \left( \frac ih f + X_0 f \right)
	\right. \\ \left.
	- g^{kj} \Gamma^0_{kj}  \left( \frac ih f + X_0 f \right)
	+2 g^{0j}\Gamma_{0j}^k  X_k f  +  \Delta_{\Gamma^{\perp}}f \right].
\end{multline}

\begin{rema}
One may initially be inclined to use the Ansatz of the original
Gaussian beam from Ralston
\cite{Ral-beam}, which is
\begin{eqnarray*}
e^{i \psi (x)/h} (a_0 + a_1 h + \dots + a_N h^N),
\end{eqnarray*}
the standard geometric optics quasimode construction.  
After all, Ralston is able to make very nice use of the geometry to construct a phase function of the form $i/h (x_0 + \tfrac12 x' B(x_0) x')$
with $\Im B(x_0) >0$ (for $x_0 \neq 0,$ vanishing otherwise).
In such a regime, however, the non-linear term in the Schr\"odinger equation \eqref{nls}
vanishes to infinite order in $h.$
Thus while such a solution always exists, it fails to capture the effects of the nonlinearity that we are interested in.
\end{rema}

We analyze \eqref{E:DeltaF} by applying the operator $T_{h,s}^{-1}$ in
the variables transversal to $\Gamma$.  We normalize everything in the
$L^2$ sense, so we take here $s=0$.  Let $z = h^{-1/2}x'$ and set
$v(x_0, z, h) = T_{h,0}^{-1} f(x_0, z, h) = h^{n/4 } f(x_0, h^{1/2}z, h)$.
Notice that the distance to the geodesic $r = |x'|$ is scaled to $\rho
= |z| = h^{-1/2} r$, as described above.  In particular, now
\begin{equation}
\label{E:v-WF}
\sWF_{h,0, 1/2} v \subset \{ | z | \leq h^{-\delta} \epsilon, \,\, | \zeta | \leq
h^{1/2} \epsilon \},
\end{equation}
if $\zeta$ is the (semiclassical) Fourier dual to $z$.  
We conjugate \eqref{E:DeltaF} to get 
\begin{multline*}
	T_{h,0}^{-1} \Delta T_{h,0} T_{h,0}^{-1} F
	= e^{ix_0/h} \left[
	g^{00} \left( -\frac1{h^2} v + \frac{2i}h X_0 v + X_0X_0 v \right)
	+ 2 g^{k0}  h^{-1/2} \partial_{z_k} \left( \frac ih v + X_0 v \right)
	\right. \\ \left.
	- g^{kj} \Gamma^0_{kj}  \left( \frac ih v + X_0 v \right)
	+2 g^{0j}\Gamma_{0j}^k  h^{-1/2}\partial_{z_k} v  +  h^{-1} \Delta_{\Gamma^{\perp}}v \right],
\end{multline*}
where the metric components and Christoffel symbols are evaluated at
$(x_0, h^{1/2}z )$.  
On the other hand, from \eqref{LocalMet}, expanding in Taylor
polynomials, we know that
\begin{equation}
\begin{split}
	g^{00} (x) &= 1 + R_2(x) + R_3(x) + R_4(x) + \mathcal{O}(r^5)
        \notag \\
	& = 1 + hR_2(z) + h^{3/2}R_3(z) + h^2 R_4(z) + \mathcal{O}(h^{5/2}\rho^5) \\
	g^{0k}(x) &= h \tilde{g}^{0k}_2(z) + \O(h^{3/2} \rho^3) \notag
        \\
	g^{jk} (x) &= \delta_{jk} +
        \mathcal{O}(h\rho^2) \label{geomexps} \\
	\Gamma^0_{jk}(x) &= h^{1/2}\widetilde\Gamma^0_{jk1} (z)+ h
        \widetilde\Gamma^0_{jk2}(z) + \mathcal{O}(h^{3/2} \rho^3)
        \notag \\
        \Gamma^k_{0j}(x) & = h^{1/2} \widetilde\Gamma^k_{0j1} (z) +
        \O(h \rho^2), \notag
\end{split}
\end{equation}
for some smooth functions $R_\ell$, $\tilde{g}^{0k}_\ell$,
$\widetilde{\Gamma}^\alpha_{jk\ell}$ homogeneous of degree $\ell$ respectively.  Hence
\begin{multline}
\label{E:TDelta}
	T_{h,0}^{-1} \Delta T_{h,0} T_{h,0}^{-1} F
	= e^{ix_0/h} \left[
	- \frac 1{h^2}v - \frac1h R_2(z)v - \frac1{h^{1/2}} R_3(z) v -
        R_4(z) v
	+ \frac{2i}h X_0 v
	+ \frac {2i}{h^{1/2}} \tilde{g}^{k0}_2 (z)\partial_{z_k}v
	\right. \\ \left.
	- \frac {i}{h^{1/2}} g^{jk} \widetilde\Gamma^0_{jk1} (z) v
        -i g^{jk} \widetilde\Gamma^0_{jk2} (z) v
	+ \frac1h \Delta_{\Gamma^{\perp}}v \right] + Pv,
\end{multline}
where $P$ contains the remaining terms from the Taylor expansion.  Let us record the following
Lemma.
\begin{lemma}
\label{L:P-error-est}
The operator $P$ has the following expansion:
\begin{eqnarray*}
P & =&  \O(h^{1/2} | z |^5) + X_0 X_0 + \O(h | z |^2 )
h^{-1/2} \partial_{z_k} X_0 + \O(|z|^3) \partial_{z_k} + \O(h^{1/2} |
z |^3) \\
&& + \O(h^{1/2} | z |) X_0 + \O(h^{3/2} | z |^{3} )
h^{-1/2} \partial_{z_k}.
\end{eqnarray*}
In particular, if $v$ satisfies \eqref{E:v-WF}, then 
\[
\| P v \|_{L^2} \leq C \| v \|_{L^2} + C \| X_0 X_0 v \|_{L^2} + C
h^{1/2-\delta} \| X_0 v \|_{L^2}.
\]

\end{lemma}

\begin{rema}
We will show later that for the particular choice of $v$ we construct,
the operators $X_0$ and $X_0^2$ are bounded operators, so that the
error $Pv$ is bounded in $L^2$ by $v$ (see Remark \ref{rem:x0bdd}).
\end{rema}

For the purposes of exposition, let us then assume for now that the
term $Pv$ is bounded and proceed (this will be justified later).  Applying $T_{h,0}^{-1}$ to the equation $-\Delta F = \lambda F -
\sigma |F|^p F$ yields
\begin{align*}
-T_{h,0}^{-1}\Delta T_{h,0}  T_{h,0}^{-1} F & = \lambda T_{h,0}^{-1} F -
\sigma T_{h,0}^{-1} (|F|^p F) \\
& = \lambda T_{h,0}^{-1} F -
\sigma h^{-pn/4}|T_{h,0}^{-1} F|^p T_{h,0}^{-1} F,
\end{align*}
so that multiplying \eqref{E:TDelta} by $h e^{-ix_0/h}$, we get 
\begin{multline} \label{EqToSolve}
	(2iX_0 + \Delta_{\Gamma^{\perp}} - R_2(z))v \\ 
	= \frac{1-h^2\lambda}{h}v 
	+h^{1/2} \left( i g^{jk} \widetilde\Gamma^0_{jk1}  -
          2i\tilde{g}^{k0}(z)\partial_{z_k} + R_3(z) \right)v 
+ h  (R_4(z) +i g^{jk} \tilde{\Gamma}^0_{jk2}) v \\
	+ h^{1-pn/4}\sigma |v|^p v
	+ he^{-ix_0/h} Pv,
\end{multline}
where $P$ is the same as above.

\begin{rema}
In order to ensure that the nonlinearity appears here as a lower order
term, we require
\begin{equation}
\label{E:p-rest1}
q:=1-pn/4>0, \text{ or } p < \frac{4}{n} = \frac{4}{d-1}
\end{equation}
as stated in the theorems. 
\end{rema}

We want to think of the left hand side as similar to a time-dependent harmonic oscillator where $x_0$ plays the role of the time variable. 

Let $ q = 1-pn/4$, $0 < q < 1$.  We would like to assume that $v$ and
$\lambda$ have expansions in $h^{q}$, however the spreading of
wavefront sets due to the nonlinearity allows us to only take the
first two terms when $0 < q \leq 1/2$ and the first three terms
otherwise.

{\bf Case 1: $0 < q \leq 1/2$.}  Assume that $v$ has a two-term expansion
\[
v = v_0 + h^q v_1
\]
 and moreover that 
 there exist
$E_k$, $k=0,1$, satisfying 
\begin{equation*}
	\frac{1-h^2\lambda}{h} = E_0 + h^{q}E_1 + \O(h^{2q}).
\end{equation*}
Since $q \leq
1/2$, then the $\O(h^{1/2})$
term in \eqref{EqToSolve} is of equal or lesser order than the nonlinear term,
and substituting into \eqref{EqToSolve} we get the following equations
according to powers of $h$:
\begin{equation}\label{SystemEqs}
\begin{split}
h^0: \quad	(2iX_0 + \Delta_{\Gamma^{\perp}} - R_2(z) - E_0) v_0 &= 0,\\
h^q : \quad	(2iX_0 + \Delta_{\Gamma^{\perp}} - R_2(z) - E_0-h^q E_1) v_1
&= E_1v_0 + \sigma | v_0 |^p v_0 + h^{1/2-q}L v_0 , 
\end{split}
\end{equation}
where
\[
L v_0 = \left( i g^{jk} \widetilde\Gamma^0_{jk1}  -
          2i\tilde{g}^{k0}(z)\partial_{z_k} + R_3(z)  + h^{1/2} R_4(z)+ih^{1/2}g^{jk}\tilde{\Gamma}^{0}_{jk2}
        + h^{1/2} e^{-ix_0/h} P \right) v_0 .
\]
We will show the error terms are $\O(h^{2q})$ in the appropriate
$H^s$ space.  See \S \ref{S:q-leq-half}.

{\bf Case 2: $1/2 < q < 1$.}  In the case $q > 1/2$, the $\O(h^{1/2})$ term becomes potentially
larger than the nonlinearity, so we take three terms in the expansions
of $v$ and $E=\frac{1-h^2\lambda}{h}$:
\[
v = v_0 + h^{1/2} v_1 + h^q v_2, \,\,\, E = E_0 + h^{1/2} E_1 + h^q
E_2 + \mathcal{O} (h).
\]
We then want to solve
\begin{equation}
\label{SystemEqs2}
\begin{split}
h^0: \quad	(2iX_0 + \Delta_{\Gamma^{\perp}} - R_2(z) - E_0) v_0
&= 0,\\
h^{1/2}: \quad 	(2iX_0 + \Delta_{\Gamma^{\perp}} - R_2(z) - E_0) v_1 &
=
E_1 v_0 + L v_0 \\
h^q : \quad	(2iX_0 + \Delta_{\Gamma^{\perp}} - R_2(z) - E_0 -
h^{1/2} E_1 -h^q E_2) v_2
&= E_2 v_0 + h^{1-q} E_1 v_1 + h^{1/2} E_2 v_1 \\
& + \sigma | v_0 |^p v_0 + h^{1-q} L v_1.
\end{split}
\end{equation}
In this case we will show the error is $\O(h^{1/2 + q})$ in the
appropriate $H^s$ space.  See \S \ref{S:q-g-half}.

\section{Quasimodes}
\label{S:quasimodes}

We begin by approximately solving the $h^0$ equation by undoing our
previous rescaling.  That is, let $w_0 (x_0, x, h) = T_{h.0} v_0(x_0,
x, h) = h^{-n/4} v_0(x_0, h^{-1/2} x, h)$, and conjugate the $h^0$
equation by $T_{h,0}$ to get:
\begin{align*}
0 & = T_{h,0} (2iX_0 + \Delta_{\Gamma^{\perp}} - R_2(z) - E_0) T_{h,0}^{-1}
T_{h,0} v_0 \\
& =  (2iX_0 + h\Delta_{\Gamma^{\perp}} - h^{-1} R_2(x) - E_0) w_0,
\end{align*}
where now the coefficients in $\Delta_{\Gamma^\perp}$ are independent
of $h$, and we have used the homogeneity of $R_2$ in the $x$
variables.  Multiplying by $h$, we have the following equation: 
\begin{equation}
\label{E:sc-orbit}
(2ihX_0 + h^2 \Delta_{\Gamma^{\perp}} -  R_2(z) - hE_0) w_0 = 0,
\end{equation}
Hence, \eqref{E:sc-orbit} is a semiclassical equation in a fixed
neighbourhood of $\Gamma$ with symbols in the $h^0$ calculus (i.e. no
loss upon taking derivatives).  The principal symbol of the operator
in \eqref{E:sc-orbit} is
\[
p = \tau - | \zeta |^2_{\tilde{g}} - R_2(z),
\]
where $\tilde{g}$ is the metric in the transversal directions to
$\Gamma$, and $\tau$ is the dual variable to the $x_0$ direction.  If
we let $\tilde{\Gamma}$ be the (unit speed) lift of $\Gamma$ to
$T^*M$, and if $\exp(s H_p )$ is the Hamiltonian flow of $p$, then
\[
\Gamma = \{ \zeta = z = 0, x_0 = s \in \reals/ \ZZ \}.
\]
Since, in the transversal directions, $p$ is a negative definite
quadratic form, the linearization of the Poincar\'e map $S$ is easy to
compute:
\[
S = \exp (H_{q}),
\]
where
\[
q = -|\zeta|^2_{\tilde{g}} - R_2(z),
\]
so that
\[
H_q = -2 a(x_0, z)^{j,k}\zeta_j \partial_{z_k} + 2 b^{j,k}(x_0) z_j \partial_{\zeta_k},
\]
where $a$ and $b$ are symmetric, positive definite matrices.
Linearizing $S$ about, say, $x_0$ and $z = \zeta = 0$ we get that
$dS(0,0)$ has all eigenvalues on the unit circle, in complex conjugate
pairs.  
That is, $\Gamma$ is still a periodic elliptic orbit of the classical
flow of $p$.

Since $p$ is defined on a fixed scale, we can glue $p$ together with
an operator which is elliptic at infinity so that $p$ is of
real principal type so that we can apply Theorem \ref{T:quasi-ch} in
the appendix to construct linear quasimodes.  Note that since we have
quasi-eigenvalue of order $\O(h)$, Theorem \ref{T:quasi-ch} 
implies the quasimodes are concentrated on a scale $|z| \leq h^{1/2}$,
$| \zeta | \leq h^{1/2}$.

This is made precise in the following proposition.

\begin{prop}
\label{P:solution-to-0-eqn}
There exists $w_0\in L^2$, $\| w_0 \|_{L^2} = 1$, and ${E}_0 = \O(1)$ 
such that
\[
(2ihX_0 + h^2 \Delta_{\Gamma^{\perp}} -  R_2(z) - hE_0)w_0 =
e_0(h).
\]
Here the error $e_0(h) = \O(h^\infty) \in L^2$ (or in any other seminorm), and 
$w_0$ has
$h$-wavefront set sharply localized on $\Gamma$ in the sense that if
$\phi \in \Psi^0_{1/2 - \delta, 1/2 - \delta}$ is $1$ near $\Gamma$, then for any $\delta>0$, $\phi
w_0 = w_0 + \O(h^\infty)$, and if $\delta = 0$, $\|
\phi w_0 \| \geq c_0 \| w_0 \|$ for some positive
$c_0$ depending on the support of $\phi$.  

Moreover, $w_0 \in H^\infty(M)$ and satisfies the estimate
\[
\| w_0 \|_{H^s(M)} = \O(h^{-s/2} ).
\]
\end{prop}

\begin{proof}
The construction follows from Theorem \ref{T:quasi-ch},
and is well-known in other settings, see for instance \cite{Ral-beam},
\cite{Bab-geo}, and \cite{CaPo-quasi}.
To get the sharp localization, apply Lemmas \ref{L:comm-1} and
\ref{L:comm-2} from the appendix to get the localization on $w_0$.
Once we know that $w_0$ is so localized, we can replace $w_0$ with
$\phi w_0$, where $\phi \in \Psi^0_{1/2 - \delta}$ is as in the
proposition.  Then $\phi w_0$ satisfies
\[
(2ihX_0 + h^2 \Delta_{\Gamma^{\perp}} -  R_2(z) - hE_0) \phi w_0 =
\tilde{e}_0(h),
\]
where 
\[
\tilde{e}_0(h) = \phi e_0(h) + [(2ihX_0 + h^2 \Delta_{\Gamma^{\perp}}
-  R_2(z) ), \phi] w_0.
\]
But now $\phi e_0(h) = \O(h^\infty)$, while the
commutator is supported outside of an $h^{1/2-\delta}$ neighbourhood
of $\Gamma$, so by the localization of $w_0$ is $\O(h^\infty)$ and
localized in a slightly larger set on the scale $h^{1/2 - \delta}$.
\end{proof}

Now recalling $v_0 = T_{h,0}^{-1} w_0$, then $v_0$ satisfies
\begin{align*}
(2iX_0 + & \Delta_{\Gamma^{\perp}} - R_2(z) - E_0) v_0 \\
& = h^{-1} T_{h,0}^{-1} (2 i hX_0 + h^2 \Delta_{\Gamma^{\perp}} -
R_2(z) - E_0) T_{h,0} {v}_0 \\
& = h^{-1} T_{h,0}^{-1} (2 i hX_0 + h^2 \Delta_{\Gamma^{\perp}} -
R_2(z) - E_0) w_0\\
& = h^{-1} T_{h,0}^{-1} e_0(h).
\end{align*}
The error $h^{-1} T_{h,s}^{-1}e_0(h) = \O(h^\infty)$ in any seminorm
still, but the function $v_0$ is now
localized on a scale $h^{-\delta}$ in space.  That is,
\[
\sWF_{h, 0, 1 - \delta} v_0 \subset \{ | x | \leq \epsilon
h^{-\delta}, | \xi | \leq \epsilon h^{1 - \delta} \}.
\]

\subsection{The inhomogeneous equation}
We are now in a position to solve the lower order inhomogeneous
equations in \eqref{SystemEqs}.  The quasimode $v_0$ has
been constructed as a ``Gaussian beam'' (see \cite{Ral-beam}); it is a harmonic oscillator
eigenfunction extended in the $x_0$ direction by the {\it quantum
monodromy operator} from \cite{SjZw-mono}, which
is defined in \eqref{E:mondef} below.  From this construction, the boundedness of
the error term $Pv_0$ as stated in Lemma \ref{L:P-error-est}.  In what
follows we construct $v_1$ in the case $0 \leq q \leq \half$ and
$v_1,v_2$ in the case $\half < q < 1$ and show that a similar bound holds
for $h^q Pv_1$ and $h^\half P v_1 + h^q P v_2$ respectively.

We want to solve 
\begin{equation}
\label{E:hq}
(2iX_0 + \Delta_{\Gamma^{\perp}} - R_2(z) - \tilde{E}_0) v_1
= E_1v_0 +G,
\end{equation}
where $\tilde{E}_0= E_0 + h^\eta E_1$ for some $\eta>0$, $E_0$ and $v_0$ have been fixed by the previous construction,
and $G = G(x_0, x)$ is a given function (periodic in $x_0$) with
\[
\sWF_{h,0,1/2} G \subset \{ | x | \leq \epsilon h^{-\delta}, | \xi |
\leq \epsilon h^{1/2} \}.
\]
Note, the localization of $G$ follows from the nonlinearity as well as geometric multipliers in the
operator $L$, see Appendix \ref{A:harmosc} and \eqref{E:v-WF}.

Conjugating by $T_{h,0}$ as before, we
get the equation
\begin{equation}
\label{E:hq2}
(2ihX_0 + h^2\Delta_{\Gamma^{\perp}} - R_2(x) - h\tilde{E}_0) w_1
=G_2 ,
\end{equation}
where
\[
G_2 = h T_{h,0}(E_1 v_0 + G).
\]
We observe then that
\[
\sWF_{h, 1/2 - \delta, 0} G_2 \subset \{ | x | \leq \epsilon h^{1/2 -
  \delta}, | \xi | \leq \epsilon  \}.
\]

Specifically, let $M(x_0)$ be the deformation family to
the quantum monodromy operator defined as the solution to:
\begin{eqnarray}
\label{E:mondef}
\left\{ \begin{array}{c}
(2ihX_0 + h^2\Delta_{\Gamma^{\perp}} - R_2(x) - h\tilde{E}_0) M(x_0) = 0, \\
M(x_0) = \id,
\end{array} \right.
\end{eqnarray}
which exists microlocally in a neighbourhood of $\Gamma \subset
T^*X$ (for further discussion see also
Appendix \ref{quasi} and references therein).  By the Duhamel formula, we write
\[
w_1 = M(x_0) w_{1,0} + \frac{i}{h} \int_{0}^{x_0} M(x_0) M(y_0)^* {G}_2
(y_0, \cdot) d y_0.
\]
We have to choose $w_{1,0}$ and $E_1$ (implicit in $G_2$) in such a
fashion to make $w_1$ periodic in $x_0$; in other words, to solve the
equation (approximately) around $\Gamma$.  Let $L$ be the primitive period of
$\Gamma$, so that $x_0 = 0$ corresponds to $x_0 = L$.  Then we require
\[
w_1(L, \cdot) = w_1(0, \cdot),
\]
or
\[
w_{1,0} = M(L) w_{1,0} + \frac{i}{h} \int_0^L M(L) M(y_0)^*
{G}_2(y_0, \cdot) d y_0.
\]
In other words, we want to be able to invert the operator $(1 -
M(L))$.  The problem is that $w_{0,0}: = w_0 (0, \cdot) = T_{h,0} v_0(0, \cdot)$ is in the kernel of
$(1 - M(L))$, so we need to choose $E_1$ in such a fashion to kill the
contribution of ${G}_2$ in the direction of $w_{0,0}$.  

Recall that 
\begin{align*}
G_2 & = hT_{h,0} (E_1 v_0 + G   ) \\
& = h (E_1 w_0 + \tilde{G}),
\end{align*}
where 
\[
\tilde{G} = T_{h,0}G.
\]
We want to solve microlocally 
\begin{align}
(1 - M(L)) w_{1,0} & = \frac{i}{h} \int_0^L M(L) M(y_0)^*
(h (E_1 w_0 + \tilde{G})) d y_0 \notag \\
& = i \int_0^L M(L) M(y_0)^*
(E_1 M(y_0) w_{0,0} + \tilde{G}) d y_0 \notag \\
& = i L M(L) E_1 w_{0,0} + 
i \int_0^L M(L) M(y_0)^*
 \tilde{G} d y_0 .\label{E:RHS-orth}
\end{align}
Let
\[
E_1 = - \frac{1}{L} \lll \int_0^L M(y_0)^*
\tilde{G} d y_0 , w_{0,0} \rrr,
\]
so that \eqref{E:RHS-orth} is orthogonal to $w_{0,0}$.  

If we denote
\[
L^2_{w_{0,0}^\perp} = \{ u \in L^2 : \lll u, w_{0,0} \rrr = 0 \},
\]
then 
by the nonresonance assumption (since $E_0 + h^\eta E_1$ is a small
perturbation of $E_0$), and the fact that $M(L)$ is unitary on
$L^2$, $(I -M(L))^{-1}$ is a bounded operator (see \cite{Chr-QMNC})
\[
(1 - M(L))^{-1} : L^2_{w_{0,0}^\perp} \to L^2_{w_{0,0}^\perp}.
\]
Hence
\[
w_{1,0} = (1 - M(L))^{-1} \frac{i}{h} \int_0^L M(L) M(y_0)^*
{G}_2(y_0, \cdot) d y_0
\]
satisfies
\[
\| w_{1,0} \|_{L^2 } \leq C h^{-1} \left\| \int_0^L  M(y_0)^*
{G}_2(y_0, \cdot) d y_0 \right\|_{L^2} \leq C L^{1/2} h^{-1} \| G_2
\|_{L^2(x_0) L^2}
\]
and
\[
w_{1,0} \in L^2_{w_{0,0}^\perp}.
\]
Furthermore, we have the estimate
\[
\| w_{1,0} \|_{\dH^s} \leq C h^{-1} h^{-s/2} \left\| \int_0^L  M(y_0)^*
{G}_2(y_0, \cdot) d y_0 \right\|_{L^2} \leq C L^{1/2} h^{-1-s/2} \| G_2
\|_{L^2(x_0) L^2},
\]
that is, the $\dH^s$ norm is controlled, but not by the homogeneous
Sobolev norm.

We have proved the following Proposition, which follows simply from
tracing back the definitions.

\begin{prop}
\label{P:inhomog}
Let $v_0$ be as constructed in the previous section, and let $G \in H^s$ for $s\geq 0$ sufficiently large satisfy
\[
\sWF_{h,0,1/2} G \subset \{ | x | \leq \epsilon h^{-\delta}, | \xi |
\leq \epsilon h^{1/2} \}.
\]
Then for any $\eta>0$, there exists $v_1 \in L^2$ and $E_1 = \O(\| G \|_{L^2})$ such
that
\[
(2iX_0 + \Delta_{\Gamma^{\perp}} - R_2(z) - E_0 - h^\eta E_1) v_1
= E_1v_0 +G,
\]
and moreover
\begin{eqnarray*}
\| v_1 \|_{\dH^s} \leq C (\| G \|_{H^s} + \| v_0 \|_{H^s} ).
\end{eqnarray*}
\end{prop}

\begin{rema}
\label{rem:x0bdd}
We note here that by construction of $v_1$ we have implicitly
microlocalized into a periodic tube in the $x_0$ variable.  Using the fact that the Quantum
Monodromy Operator is a microlocally unitary operator (see \cite{SjZw-mono,Chr-QMNC}), the bound
\begin{eqnarray*}
\| X_0^j v_1 \|_{L^2} \lesssim \| v_1 \|_{L^2}
\end{eqnarray*}
follows easily for any $j$, which is required for proof of Lemma
\ref{L:P-error-est}.
\end{rema}

 \subsection{Construction of quasimodes in the context of Theorem \ref{T:2}}
 We now have all the tools to construct  the quasimodes which will be used to prove Theorem \ref{T:2}. Let $d\geq 2$ and $0<p<4/(d-1)$. As previously, denote by $q=1-p(d-1)/4$. The main results of this part will be stated in Propositions \ref{PQ1} and \ref{PQ2}.
\subsubsection{The case $0 < q \leq 1/2$, i.e. $\frac2{d-1}\leq p< \frac4{d-1}$}
\label{S:q-leq-half}
In this subsection, we see how to apply Proposition \ref{P:inhomog} in
the case $0 < q \leq 1/2$.  As described previously, in this case, the
nonlinearity is the next largest term, and we have only one
inhomogeneous equation so solve (see \eqref{SystemEqs}).

According to Propostion \ref{P:v-symbolic} and Corollary \ref{C:v-symbolic} from Appendix \ref{A:harmosc}, if 
\[
G = \sigma | v_0 |^p v_0 - h^{1/2-q}L v_0
\]
is the nonlinear term on the right-hand side, then $G$ is sharply
localized in space but weakly localized in frequency.  That is, if
$\chi \in \Ci_c( T^*X)$ is equal to $1$ in a neighbourhood of
$\Gamma$, then for any $0 \leq \delta < 1/2$ and any $0 \leq \gamma
\leq 1$, 
\[
\chi(h^\delta x, h^{1-\gamma} D_x) G(x_0, x) = G(x_0, x) + E,
\]
where for any $0 \leq r \leq 3/2$, 
\[
\| E \|_{\dH^r} \leq C h^{(1 - \gamma)(3/2 + p -r)}.
\]
We are interested in the case where $\gamma = 1/2$, since in that case
$G$ is weakly concentrated in frequencies comparable to $h^{1/2}$, so
by cutting off, satisfies the assumptions of Proposition
\ref{P:inhomog}.  That is, take $\gamma = 1/2$, and replace $G$ with
$\tilde{G} = \chi(h^\delta x, h^{1-\gamma} D_x) G$, and apply Proposition
\ref{P:inhomog} to get $v_1$ and $E_1$ satisfying
\[
(2iX_0 + \Delta_{\Gamma^{\perp}} - R_2(z) - E_0 - h^q E_1) v_1
= E_1v_0 +\tilde{G},
\]
or in other words
\[
(2iX_0 + \Delta_{\Gamma^{\perp}} - R_2(z) - E_0 - h^q E_1) v_1
= E_1v_0 + \sigma| v_0|^p v_0 + h^{1/2-q}L v_0+ \tilde{Q}_1,
\]
where 
\[
\| \tilde{Q}_1 \|_{\dH^r} \leq C h^{(1/2)(3/2 + p -r)}.
\]

Now letting $v = v_0 + h^q v_1$ and $E = E_0 + h^q E_1$, we have solved
\begin{align*}
(2iX_0 + \Delta_{\Gamma^{\perp}} - R_2(z) -E) v
& =h^q \sigma| v_0|^p v_0 + h^{1/2}L v_0+ h^q \tilde{Q}_1 \\
& = h^q \sigma |v|^p v + h^{1/2} L v + \tilde{Q}_2,
\end{align*}
where
\[
\tilde{Q}_2 = h^q \tilde{Q}_1 - h^{1/2+q} Lv_1 +\O(h^{2q} | v |^{p+1} ).
\]
The remainder $\tilde{Q}_2$ satisfies
\[
\| \tilde{Q}_2 \|_{\dH^s} \leq C h^{-s/2 + \min \{ 2q, 1/2 + q, q + 3/4 + p/2 \} } =
Ch^{-s/2 + 2q} ,
\]
since $q\leq 1/2$.  
Recalling the definitions, $\phi = e^{ix_0 /h} T_{h,0} v$ satisfies
\begin{itemize}
\item[(i)]
\[
\sWF_{h, 1/2 - \delta, 0} \,\phi \subset \{ | x | \leq \epsilon h^{1/2 -
  \delta}, | \xi | \leq \epsilon \};
\]
\item[(ii)]
\[
\| \phi \|_{L^2} \sim 1, \qquad \| D_{x_0}^\ell \phi \|_{L^2} \sim h^{-\ell},
\]
and
\[
\| D_{x'}^\ell \phi \|_{L^2} \leq C h^{-\ell/2} ;
\]
\item[(iii)]
\begin{align*}
(\Delta + \lambda)\phi & = h^{-1} e^{i x_0/h} T_{h,0} h e^{-i x_0/h} T_{h,0}^{-1}
\Delta T_{h,0} T_{h,0}^{-1} \phi \\
& = h^{-1} e^{i x_0/h} T_{h,0} (2 i X_0 + \Delta_{\Gamma^\perp} - R_2
-E 
- h^{1/2} L) v \\
& = h^{-1} e^{i x_0/h} T_{h,0} ( \sigma h^q | v |^p v +
Q_2),
\end{align*}
or
\[
(\Delta + \lambda)\phi = \sigma | \phi |^p \phi + h^{\alpha(p)}Q,
\]
where $Q = h^{-1} e^{i x_0/h} T_{h,0} \tilde{Q}_2$ satisfies
$\| Q \|_{\dH^s} \leq C h^{-s}$
and where 
\begin{equation}\label{error1}
\alpha(p)  := -1+ 2q   =1 - p \left( \frac{d-1}{2} \right).
\end{equation}
\end{itemize}

We now sum up what we have proven in a proposition. Consider the objects we have just defined : $v=v_{0}+h^{q}v_{1}$, $\phi_{h}=e^{ix_{0}/h}T_{h,0}v$ and $\lambda(h)=h^{-2}-E_{0}h^{-1}-E_{1}h^{-1+q}$. Then we can state
\begin{proposition}\label{PQ1}
Let $2/(d-1)\leq p<4/(d-1)$ and $\alpha(p)$ be given by \eqref{error1}. Then the  function $\phi_{h}$ satisfies the equation
 \begin{equation*}
 \big(\Delta+\lambda(h)\big)\phi_{h}=\sigma |\phi_{h}|^{p}\phi_{h}+h^{\alpha(p)}Q(h),
 \end{equation*}
 where ${Q}(h)$ is an error term which satisfies $\|{Q}(h)\|_{\dot{H}^{s}}\lesssim h^{-s}$, for all $s\geq 0$.
\end{proposition}

\subsubsection{The case $1/2 < q \leq 1$, i.e. $0\leq p< \frac2{d-1}$}
\label{S:q-g-half}
 We again construct
$v_0$ as a Gaussian beam using the quantum monodromy operator.  We
then set 
\[
G_1 = E_1 v_0 + L v_0 ,
\]
which is smooth with compact wavefront set contained in the wavefront
set of $v_0$, so no phase space cutoff is necessary to apply the
inhomogeneous argument to get $v_1$ with wavefront set contained in
the wavefront set of $v_0$.  

Now let
\[
G_2 = E_2 v_0 + h^{1-q} E_1 v_1 + h^{1/2} E_2 v_1+ \sigma | v_0 |^p
v_0 + h^{1-q} L v_1,
\]
and solve for $v_2$ as in the previous subsection.  This is possible
since $v_1$ is orthogonal to $v_0$ by construction.  
We then have 
\[
(2iX_0 + \Delta_{\Gamma^{\perp}} - R_2(z) - E_0 - h^q E_1) v_2
= E_2 v_0 + h^{1-q} E_1 v_1 + h^{1/2} E_2 v_1+ \sigma | v_0 |^p
v_0 + h^{1-q} L v_1+ \tilde{Q}_1,
\]
where 
\[
\| \tilde{Q}_1 \|_{\dH^r} \leq C h^{(1/2)(3/2 + p -r)}.
\]
Letting $v = v_0 + h^{1/2} v_1 + h^q v_2$ and $E = E_0 + h^{1/2} E_1 +
h^q E_2$, we have solved
\begin{align*}
(2iX_0 + \Delta_{\Gamma^{\perp}} - R_2(z) -E) v
& =h^q \sigma| v_0|^p v_0 + h^{1/2}L v_0 + hL v_1 + \tilde{Q}_1 \\
& = h^q \sigma |v|^p v + h^{1/2} L v + \tilde{Q}_2,
\end{align*}
where
\[
\tilde{Q}_2 = h^q \tilde{Q}_1 - h^{1/2 + q} Lv_2 +\O(h^{1/2 + q} | v |^{p+1} ).
\]
We now have the remainder estimate
\[
\| \tilde{Q}_2 \|_{\dH^s} \leq C h^{-s/2+ \min \{ 1/2 + q, q + 3/4 + p/2 \} } = C
h^{-s/2 + 1/2 + q}.
\]

Recalling the definitions, $\phi: = e^{ix_0 /h} T_{h,0} v$ satisfies
\begin{itemize}
\item[(i)]
\[
\sWF_{h, 1/2 - \delta, 0} \,\phi \subset \{ | x | \leq \epsilon h^{1/2 -
  \delta}, | \xi | \leq \epsilon \};
\]
\item[(ii)]
\[
\| \phi \|_{L^2} \sim 1, \qquad \|D_{x_0}^\ell \phi \|_{L^2} \sim h^{-\ell},
\]
 and
\[
\| D_{x'}^\ell \phi \|_{L^2} \leq C h^{-\ell/2} ;
\]
\item[(iii)]
\begin{align*}
(\Delta + \lambda)\phi & = h^{-1} e^{i x_0/h} T_{h_0} h e^{-i x_0/h} T_{h_0}^{-1}
\Delta T_{h_0} T_{h_0}^{-1} \phi \\
& = h^{-1} e^{i x_0/h} T_{h_0} (2 i X_0 + \Delta_{\Gamma^\perp} - R_2
-E 
- h^{1/2} L) v \\
& = h^{-1} e^{i x_0/h} T_{h_0} ( \sigma h^q | v |^p v +
\tilde{Q}_2),
\end{align*}
or
\[
(\Delta + \lambda)\phi = \sigma | \phi |^p \phi + h^{\alpha(p)}Q,
\]
where $Q = h^{-1} e^{i x_0/h} T_{h_0} \tilde{Q}_2$ satisfies
$\| Q \|_{\dH^s} \leq C h^{-s}$ and 
where
\begin{equation}\label{error2}
\alpha(p):= -1/2 +q=\frac{1}{2} - p \left( \frac{d-1}{4} \right). 
\end{equation}
\end{itemize}

Once again, by construction we have 
\begin{eqnarray*}
\| X_0^j v_j \|_{L^2} \lesssim \| v_j \|_{L^2}
\end{eqnarray*}
for $j =0,1,2$.

Consider $v=v_{0}+h^{\frac12}v_{1}+h^{q}v_{2}$, $\phi_{h}=e^{ix_{0}/h}T_{h,0}v$ and $\lambda(h)=h^{-2}-E_{0}h^{-1}-E_{1}h^{-\frac12}-E_{2}h^{-1+q}$ defined previously, then we have proven
\begin{proposition}\label{PQ2}
Let $0<p\leq 2/(d-1)$ and $\alpha(p)$ be given by \eqref{error2}. Then the  function  $\phi_{h}$ satisfies the equation
 \begin{equation*}
 \big(\Delta+\lambda(h)\big)\phi_{h}=\sigma |\phi_{h}|^{p}\phi_{h}+h^{\alpha(p)}{Q(h)},
 \end{equation*}
 where ${Q(h)}$ is an error term which satisfies $\|Q(h)\|_{\dot{H}^{s}}\lesssim h^{-s}$, for all $s\geq 0$.
\end{proposition}

\subsection{Higher order expansion and proof of Thereorem \ref{T:0}} \label{sec:nthorder}
\subsubsection{The two dimensional cubic equation}
We first deal with the simpler case $d=2$, $p=2$ and $s=0$. As in the previous section, we define $q=1-p(d-1)/4$, thus for this choice we have $q = \frac12$
allowing us to match powers of the asymptotic
parameters in a canonical way.  The general algorithm for any smooth
nonlinearity arising when rescaling in the appropriate $H^s$ space will follow
similarly.

Using \eqref{EqToSolve} and the Taylor expansions of the geometric
components $g^{ij}$ and $\Gamma^i_{jk}$ for $i,j,k = 0,\dots,d$ in
\eqref{geomexps}, we look for an asymptotic series solution of the
form
\begin{equation*}
v  =  v_0 + h^{\frac12} v_1 + h^{1} v_2 + \dots + h^{m \frac12}
v_{m} + \dots + h^{N \frac12} v_N + \tilde{v} ,
\end{equation*}
for $N$ sufficiently large.

Then, we have the following equations:
\begin{eqnarray*}
h^0 : && (2 i X_0 + \Delta_{\Gamma^\perp} - R_2 (z) - E_0 ) v_0 = 0 ,
\\
h^{\frac12} : && (2 i X_0 + \Delta_{\Gamma^\perp} - R_2 (z) - E_0 ) v_1  =
E_1 v_0 + (i \delta_{jk} \tilde{\Gamma}^0_{jk1} - 2i \tilde{g}^{k0}
\p_{z_k} + R_3 (z)) v_0 + \sigma |v_0|^2 v_0 , \\
h^{1} : &&  (2 i X_0 + \Delta_{\Gamma^\perp} - R_2 (z) - E_0 ) v_2  =
E_2 v_0 + E_1 v_1 + (i \delta_{jk} \tilde{\Gamma}^0_{jk1} - 2i \tilde{g}^{k0}
\p_{z_k} + R_3 (z)) v_1 \\
&&+ (i \delta_{jk} \tilde{\Gamma}^0_{jk2} + R_4) v_0 + \sigma (2 |v_0|^2 v_1 + v_0^2 \bar{v}_1),  \\
&\vdots& \\
h^{\frac{m}2} : && (2 i X_0 + \Delta_{\Gamma^\perp} - R_2 (z) - E_0 ) v_m  =
\sum_{j=0}^{m-1} E_{m-j} v_j + \sigma (\sum_{j,k,l=0}^{m-1} (c^m_{jkl}
v_j v_k \bar{v}_l)) \\
&&+ \sum_{j=0}^{m-1} (f^{\p_{z_k}}_{j,m} (z) \p_{z_k} +
f^{X_0}_{j,m} X_0 + f^1_{j,m} (z) ) v_j ,  \\
&\vdots& , \\
h^{\frac{N}2} : && (2 i X_0 + \Delta_{\Gamma^\perp} - R_2 (z) - E_0
- \sum_{j=1}^N h^{\frac{j}{2}} E_j) v_N  =
\sum_{j=0}^N E_{N-j} v_j + \sigma (\sum_{j,k,l=0}^{N-1} (c^N_{jkl} v_j
v_k \bar{v}_l)) \\
&&+ \sum_{j=0}^{N-1} (f^{\p_{z_k}}_{j,N} (z) \p_{z_k} +
f^{X_0}_{j,N} X_0+ f^1_{j,N} (z) ) v_j + P_N v,
\end{eqnarray*}
where 
\begin{eqnarray*}
f^{\p_{z_k}}_{j,m} &=& \O_N (|z|^{m-j}), \\
f^{X_0}_{j,m} &=& \O_N (|z|^{m-j}), \\
f^1_{j,m} &=& \O_N (|z|^{m-j})
\end{eqnarray*}
for $j,m = 0, \dots, N$ and
\begin{eqnarray*}
P_{N} & = & \O(h^{N/2-2} | z |^N) + X_0 X_0 + \O(h^{N/2} | z |^{N} )
h^{-1/2} \partial_{z_k} X_0 \\
&&+ \O(h^{N/2-3/2}|z|^N) \partial_{z_k} + \O(h^{N/2-1} |
z |^N) + \O(h^{N/2} | z |^N) X_0 + \O(h^{N/2} | z |^{N} )
h^{-1/2} \partial_{z_k}.
\end{eqnarray*}
Note that all constants have implicit dependence upon $N$ relating the
number of terms in the expansion at each order.  The expansion is
valid provided first of all that
\begin{eqnarray*}
\sum_{j=1}^N h^{\frac{j}{2}} E_j < E_0
\end{eqnarray*}
in order to justify the solvability of the $\O (h^{N/2})$ equation.

\begin{rema}
We note here that
in this expansion, the sign of $\sigma$ can effect the sign and value
of $E_1$, which will impact the remaining asymptotic expansion and in
particular the order of quasimode expansion possible.  It is
possible that the focussing/defocussing problem enters in to the
stability analysis of these quasimodes through this point.  
\end{rema}

Applying Proposition \ref{P:inhomog} at each asymptotic order and
bounds similar to those in Lemma \ref{L:P-error-est} at order $h^{N/2}$ as in
Section \ref{S:q-leq-half}, we have by a simple calculation that $v$
is a quasimode for the nonlinear elliptic equation with remainder $Q_N$
such that
\begin{eqnarray*}
\| Q_N \|_{\dot{H}^s} \leq C_N h^{-s-1+N/2} .
\end{eqnarray*}
As a result, for sufficiently smooth nonlinearities, one is capable of
constructing higher order asymptotic expansions and hence a quasimode
of higher order accuracy.

\subsubsection{The general case}
Let $d\geq 2$, $p\in 2\NN$ and $s\geq0$. We define here $q_{s}=1+p(s-\frac{d-1}4)$.  Assume that $p(\frac{d-1}4-s)<1$, or equivalently  that $q_{s}>0$.  Firstly, write  $w=h^{s}v$. Then $v$ has to satisfy \eqref{EqToSolve} but where the power in front of the nonlinearity is $h^{q_{s}}$.  
 Hence, we can look for $v$ and $E$ of the form
\begin{equation*}
v  = \sum_{j,\ell=0}^{N}h^{j/2+\ell q_{s} }\,v_{j,\ell}+\tilde{v} \quad \text{and}\quad  E=\sum_{j,\ell=0}^{N}h^{j/2+\ell q_{s} }\,E_{j,\ell}+\widetilde{E}.
\end{equation*}
Since $q_{s}>0$, the nonlinearity does not affect the equation giving $v_{0,0}$, and we have     
\begin{equation*}
(2 i X_0 + \Delta_{\Gamma^\perp} - R_2 (z) - E_{0,0} ) v_{0,0} = 0,
\end{equation*}
which is the same equation as before. Then, using Taylor expansions, we write all the equations, similarly to the previous case, in powers of $h^{j/2+\ell q_{s}}$. Again, we can solve each equation and obtain bounds of the solutions and of the error terms. Moreover, it is clear that we can go as far as we want in the   asymptotics, so that we can construct a $\mathcal{O}(h^{\infty})$ quasimode, and this proves Theorem \ref{T:0}. 

\section{Error estimates}

\subsection{The regular case}

In this section, we assume that $p$ is an even integer.\\

Fix an integer $k>d/2$   (the fact that $k$ is an integer is not necessary). We then define the semiclassical norm    
\begin{equation*}
\|f\|_{H^k_{h}}=\|\big(1-{h}^{2}\Delta\big)^{k/2}f\|_{L^2(M)}.
\end{equation*}
In the previous section we have shown the following : Given $\alpha \in \RR$, there exist  two functions $\varphi_{h}\in H^{k}(M)$ and ${Q}(h)\in H^{k}(M)$  and $\lambda(h) \in \RR$ so that 
\begin{equation*}
\big(\Delta+\lambda(h)\big) \varphi_{h} +\sigma |\varphi_{h}|^{p}\varphi_{h} +h^{\alpha} {Q}(h).
\end{equation*}
 Moreover, microlocally the function $\varphi_{h}$ takes the form 
\begin{equation}\label{dephi}
 \varphi_{h}(\sigma,x',h)=h^{-\frac{d-1}4+s}e^{i\sigma/h}f(\sigma,h^{-1/2}x',h),
\end{equation}
and we have  $\| {Q}(h)\|_{H^{k}_{h}}\leq C$.\\
We set $\ds u_{app}(t,\cdot)=e^{-it \lambda(h)} \varphi_{h}$. Then if we denote by $\widetilde{Q}(h):=e^{-it\lambda}Q(h)$, the following equation is satisfied
\begin{equation}\label{equapp}
i\partial_{t}  u_{\text{app}}-\Delta  u_{\text{app}}=\sigma | u_{\text{app}}|^{p} u_{\text{app}}+h^{\alpha}\widetilde{Q}(h).
\end{equation}

\begin{prop}\label{properreur}
Let $s\geq 0$. Consider the function $\phi_{h}$ given by
\eqref{dephi}. Let $u$ be the solution of
\begin{equation}\label{nlsapp}
\left\{
\begin{aligned}
&i \partial_t u-\Delta u=\sigma|u|^{p} u,\\
&u(0,\cdot)=\varphi_{h}.
\end{aligned}
\right.
\end{equation}
Assume that  $\ds \alpha>\frac{d+1}4+s+p(-\frac{d-1}4+s)$. Then there exists $C>0$ and $c_{0}>0$ independent of $h$ so that   
$$\|u- u_{\text{app}}\|_{L^{\infty}([0,T_{h}];H^{s}(M))}\leq C h^{(d+1)/4},$$
for $0\leq T_{h}\leq c_{0} \,h^{p(\frac{d-1}4-s)}\ln(\frac1h)$.
\end{prop}
This result shows that $u_{app}$ is a good approximation of $u$, provided that the quasimode $\phi_{h}$ has been computed at a sufficient order $\alpha$.
\begin{proof} Here we follow the main lines of \cite[Corollary 3.3]{Tho3}.
With the Leibniz rule and interpolation we check that for all
$f\in H^{k}(M)$ and $g\in W^{k,\infty}(M)$
\begin{equation}\label{produithh}
\|f\,g\|_{H^k_{h}}\lesssim
\|f\|_{H^k_{h}}\|g\|_{L^{\infty}(M)}+\|f\|_{L^2(M)}
\|\big(1-{h}^{2}\Delta\big)^{k/2}g\|_{L^{\infty}(M)}.
\end{equation}
Moreover, as $k>d/2$, for all $f_1,f_2\in {H^k(M)}$
\begin{equation}\label{produithh2}
\|f_1\,f_2\|_{H^k_{h}}\lesssim   h^{-d/2}\|f_1\|_{H^k_{h}}\|f_2\|_{H^k_{h}}.
\end{equation}
Let $u$ be the solution of \eqref{nlsapp} and define
$w=u-u_{\text{app}}$. Then, by \eqref{equapp}, $w$ satisfies
\begin{equation}\label{nlsapp2}
\left\{
\begin{aligned}
&i \partial_t w-\Delta w   = \sigma\big( |w+u_{\text{app}}|^{p}
  (w+u_{\text{app}})- |u_{\text{app}}|^{p} u_{\text{app}} \big)-h^{\alpha}\widetilde{Q}(h)\\
&w(0,x)=0.
\end{aligned}
\right.
\end{equation}
We expand the r.h.s. of \eqref{nlsapp2}, apply the operator
$\big(1-{h}^{2}\Delta\big)^{k/2}$ to the equation, and take
the $L^2$- scalar product with
$\big(1-{h}^{2}\Delta\big)^{k/2}w$. Then we obtain
\begin{equation}\label{nlsapp5}
\frac{\text{d}}{\text{d}t}\|w\|_{H^k_h}\lesssim \sum_{j=1}^{p+1}
 \|w^j\,u^{p+1-j}_{\text{app}}\|_{H^k_h}+h^{\alpha}.
\end{equation}
We now have to estimate the terms
$\|w^j\,u^{p+1-j}_{\text{app}}\|_{H^k_h},$
for $1\leq j\leq p+1$. From \eqref{produithh} we deduce 
\begin{equation}\label{nlsapp3}
\|w^j\,u^{p+1-j}_{\text{app}}\|_{H^k_h}\lesssim
\|w^j\|_{H^k_h}\|u^{p+1-j}_{\text{app}}\|_{L^{\infty}(M)}+\|w^j\|_{L^2(M)}\|\big(1-{h}^{2}\Delta\big)^{k/2}u^{p+1-j}_{\text{app}}\|_{L^{\infty}(M)}.
\end{equation}
By \eqref{produithh2}, and as we have
\begin{equation}\label{nlsapp4}
\|u^{p+1-j}_{\text{app}}\|_{L^{\infty}(M)}\lesssim  h^{(p+1-j)(-\frac{d-1}4+s)},\quad
\|\big(1-{h}^{2}\Delta\big)^{k/2}u^{p+1-j}_{\text{app}}\|_{L^{\infty}(M)}\lesssim h^{(p+1-j)(-\frac{d-1}4+s)},
\end{equation}
thus inequality \eqref{nlsapp3} yields
\begin{equation*}
\|w^j\,u^{p+1-j}_{\text{app}}\|_{H^k_h}\lesssim  h^{-d(j-1)/2} h^{(p+1-j)(-\frac{d-1}4+s)}\|w\|^{j}_{H^k_h}.
\end{equation*}
Therefore, from \eqref{nlsapp5} we have
\begin{equation*}
\frac{\text{d}}{\text{d}t}\|w\|_{H^k_h}\lesssim h^{p(-\frac{d-1}4+s)}\|w\|_{H^k_h}
+  h^{-dp/2}\|w\|^{p+1}_{H^k_h}+h^{\alpha}.
\end{equation*}
Observe that $\|w(0)\|_{H^k_h}=0$. Now, for times $t$ so that 
\begin{equation}\label{bootstrap}
h^{-dp/2}\|w\|^{p+1}_{H^k_h} \lesssim  h^{p(-\frac{d-1}4+s)}\|w\|_{H^k_h},
\end{equation}
i.e. $\ds \|w\|_{H^{k}_{h}}\leq C h^{(d+1)/4+s}$, we can remove the nonlinear term in \eqref{nlsapp3}, and by the
Gronwall Lemma,
 \begin{equation}\label{bootstrap2}
\|w\|_{H^k_h}\leq C h^{\alpha -p(-\frac{d-1}4+s)}e^{Ch^{p(-\frac{d-1}4+s)}t}.
\end{equation}
If $c_{0}>0$ is small enough, and $t\leq c_{0} h^{p(\frac{d-1}4-s)}\ln \frac1h$, then 
\begin{equation*}
 C h^{\alpha -p(-\frac{d-1}4+s)}e^{Ch^{p(-\frac{d-1}4+s)}t}\leq  C h^{(d+1)/4+s},
\end{equation*}
so that   inequality \eqref{bootstrap} is satisfied. By the usual bootstrap argument, we infer that  for all $$t\leq c_{0} h^{p(\frac{d-1}4-s)}\ln \frac1h$$ we have $$\ds \|w\|_{H^{k}_{h}}\leq C h^{(d+1)/4+s}.$$

 Finally, by interpolation we get   $\ds \|w\|_{H^{s}}  \leq  h^{-s}\|  w\|_{H_{h}^{k}}$, hence the result.
 \end{proof}~
 
We are now ready to complete the proof of Theorem \ref{T:1}. Consider $u_{h}$, the exact solution to \eqref{nls} with initial condition $\phi_{h}$, then by the previous proposition and the description of $u_{app}$, we can write
\begin{align}
\| u_{h} \|_{L^\infty([0,T_h]; L^2( M \setminus U_{h^{1/2-\delta
      }}))} & \leq \| u_{app} \|_{L^\infty([0,T]; L^2( M \setminus U_{h^{1/2-\delta}}))}  + \| u_{h} - u_{app} \|_{L^\infty([0,T]; L^2( M \setminus U_{h^{1/2-\delta}}))} \nonumber\\
& = \O(h^{\infty}) + \O(h^{(d+1)/4})=\O(h^{(d+1)/4}),\label{arg}
\end{align}
 which was the claim.

\subsection{The non regular case and $d=2$}~

In this section we compute the error estimate in the case of a non smooth nonlinearity in dimension $d=2$. Moreover we restrict ourselves to the case $s=0$ in \eqref{dephi} (case of an $L^{2}$-normalized initial condition).
\begin{prop}\label{properreur2}
Let $\varphi_{h}$ be the function given by \eqref{dephi} with $s=0$. Let $u$ be solution of
\begin{equation*} 
\left\{
\begin{aligned}
&i \partial_t u-\Delta u=\sigma|u|^{p} u,\\
&u(0,\cdot)=\varphi.
\end{aligned}
\right.
\end{equation*}
Let $\epsilon>0$. For  $p\in (0,4)\backslash \{1\}$, we set $T_{h}=h^{p}$, and in the case $p=1$, $T_{h}=h^{1+\epsilon}$. Then there exists $C>0$ and $\nu>0$ independent of $h$ so that   
$$\|u- u_{\text{app}}\|_{L^{\infty}([0,T_{h}];L^{2}(M))}\leq Ch^{\nu}.$$
\end{prop}

\begin{rema}
Note the difference between the results of Propositions \eqref{properreur} (when $s=0$ and $d=2$) and \eqref{properreur2}. In the first case, we have $T_{h}$ of order $h^{p/4}$, which is better than $T_{h}\sim h^{p}$ obtained in the second result. However, in this latter result, there is no restrictive condition on the size of the error term in the equation.
\end{rema}
 \begin{proof} First, we follow  the strategy of Burq-G\'erard-Tzvetkov \cite[Section 3.]{BGT-comp}
Let $0<p<4$, choose     $r>\max(p,2)$  and take  $\ds 1-\frac1r<s<1$ (there will be an additional constraint on $s$ in the sequel). Then take $q$ so that $\frac1r+\frac{1}q=\frac{1}2$ and $s_{1}=s-\frac1r$. For $T>0$ define the space
 \begin{equation*}
 Y^{s}=\mathcal{C}\big([0,T];H^{s}(M)\big)\cap L^{r}\big([0,T];W^{s_{1},q}(M)\big),
 \end{equation*}
which is endowed with the norm
\begin{equation*}
\|u\|_{Y^{s}}=\max_{0\leq t\leq T}\|u(t)\|_{H^{s}}+\|(1-\Delta)^{s_{1}/2}u\|_{L^{r}([0,T];L^{q})}.
\end{equation*}
By the Sobolev embeddings, we have $Y^{s}\subset L^{r}\big([0,T],L^{\infty}\big)$.
Now, define $w=u-u_{app}$. Then $w$ satisfies the equation 
\begin{equation}\label{nlsapp*}
\left\{
\begin{aligned}
&i \partial_t w-\Delta w   = \sigma\big( |w+u_{\text{app}}|^{p}
  (w+u_{\text{app}})- |u_{\text{app}}|^{p} u_{\text{app}} \big)-h^{\alpha(p)}\widetilde{Q}(h),\\
&w(0,x)=0,
\end{aligned}
\right.
\end{equation}
with  $\alpha(p)=1-p/2$ when $2\leq p\leq 4$ and   $\alpha(p)=1/2-p/4$ when $0\leq p\leq 2$ (see \eqref{error1} and \eqref{error2}) and $\|\widetilde{Q}(h)\|_{H^{s}}\leq Ch^{-s}$.\\[5pt]
$\bullet$ Case $0<p<1$. In \cite[Estimate (2.25)]{CaFaHa}, Cazenave, Fang and Han prove that for all $0\leq s<1$
\begin{equation}\label{est.CFH}
\big\||w+u_{app}|^{p}(w+{u_{app}})-|u_{app}|^{p}u_{app}\big\|_{H^{s}}\leq 
C\|u_{app}\|_{H^{s}}\|w\|^{p}_{L^{\infty}}+C\|w\|_{H^{s}}\big(\|u_{app}\|^{p}_{L^{\infty}}+\|w\|^{p}_{L^{\infty}}\big).
 \end{equation}
Indeed, in \cite{CaFaHa}, the estimate is not stated exactly with these indices, but the proof still holds true. Moreover, in \cite{CaFaHa}, \eqref{est.CFH}
 is proved for $x\in \RR^{d}$, but the inequality can be adapted to the case of a compact manifold thanks to a partition of unity argument.\\[5pt]
Assume that $w$ satisfies the equation
\begin{equation*}
i \partial_t w-\Delta w   =F,\quad w(0,x)=0,
\end{equation*}
then  with the Strichartz estimates of \cite{BGT-comp}, the estimate $\|w\|_{Y^{s}}\leq C\|F\|_{L^{1}_{T}H^{s}}$ holds true.
Thus, with the notation $\ds \gamma=1-\frac{p}r$, with \eqref{nlsapp*} and \eqref{est.CFH} we have
\begin{eqnarray}
\|w\|_{Y^{s}}&\leq& C\int_{0}^{T}\|u_{app}\|_{H^{s}}\|w\|^{p}_{L^{\infty}}+C\int_{0}^{T}\|w\|_{H^{s}}\big(\|u_{app}\|^{p}_{L^{\infty}}+\|w\|^{p}_{L^{\infty}}\big)+CTh^{\alpha(p)}\|\widetilde{Q}\|_{L^{\infty}_{T}H^{s}}\nonumber\\
&\leq &CT^{\gamma}\|u_{app}\|_{L_{T}^{\infty}H^{s}}\|w\|^{p}_{Y^{s}}+C\|w\|_{Y^{s}}\big(T\|u_{app}\|^{p}_{L_{T}^{\infty}L^{\infty}}+T^{\gamma} \|w\|^{p}_{Y^{s}}\big)+CTh^{\alpha(p)}\|\widetilde{Q}\|_{L^{\infty}_{T} H^{s}}\nonumber\\
&\leq &CT^{\gamma}h^{-s}\|w\|^{p}_{Y^{s}}+C\|w\|_{Y^{s}}\big(Th^{-p/4}+T^{\gamma} \|w\|^{p}_{Y^{s}}\big)+CTh^{-s+1/2-p/4}.\label{bo1}
\end{eqnarray}
 Similarly, we obtain
\begin{equation}\label{bo01}
\|w\|_{L^{\infty}_{T}L^{2}}\leq CT^{\gamma}\|w\|^{p}_{Y^{s}}+C\|w\|_{L^{\infty}_{T}L^{2}}\big(T\|u_{app}\|^{p}_{L_{T}^{\infty}L^{\infty}}+T^{\gamma} \|w\|^{p}_{Y^{s}}\big)+CTh^{1/2-p/4}.
\end{equation}
Therefore, if we define the semiclassical norm $\|\;\;\|_{Y^{s}_{h}}$ by
\begin{equation}\label{semic}
\|u\|_{Y^{s}_{h}}=h^{-s}\|u\|_{L^{\infty}_{T}L^{2}}+\|u\|_{Y^{s}},
\end{equation}
thanks to \eqref{bo1} and \eqref{bo01} we infer 
\begin{equation}\label{borne}
\|w\|_{Y^{s}_{h}}\leq CT^{\gamma}h^{-s}\|w\|^{p}_{Y_{h}^{s}}+C\|w\|_{Y^{s}_{h}}\big(Th^{-p/4}+T^{\gamma} \|w\|^{p}_{Y_{h}^{s}}\big)+CTh^{-s+1/2-p/4}.
\end{equation}
Next we use the inequality $\ds ab\leq \frac1{p_{1}}\epsilon^{p_{1}}a^{p_{1}}+\frac1{p_{2}}\epsilon^{-p_{2}}b^{p_{2}}$ which holds for $a,b,\epsilon>0$ and $\ds \frac1{p_{1}}+\frac1{p_{2}}=1$. With a suitable choice of  $\epsilon$ and $p_{1}$ (here we use that $0<p\leq 1$) we get
\begin{equation}\label{bo2}
CT^{\gamma}h^{-s}\|w\|^{p}_{Y_{h}^{s}}\leq \frac12 \|w\|_{Y_{h}^{s}}+C(T^\gamma h^{-s})^{1/(1-p)}.
\end{equation}
Now, re-inject \eqref{bo2} into \eqref{borne} and obtain
\begin{equation}\label{bo3}
\|w\|_{Y_{h}^{s}}\leq C\|w\|_{Y_{h}^{s}}\big(Th^{-p/4}+T^{\gamma} \|w\|^{p}_{Y_{h}^{s}}\big)+C(T^\gamma h^{-s})^{1/(1-p)}+CTh^{-s+1/2-p/4}.
\end{equation}
We now perform a bootstrap argument :
Fix $\epsilon >0$ and set  $T_{h}=h^{p}$. Fix $\ds 1-\frac1r<s<1-\frac{p}r$. Then  it is possible to pick $\nu>0$ small enough so that $\ds \gamma=1-\frac{p}r>s+\frac{\nu(1-p)}{p}$. Assume  that 
 \begin{equation}\label{bo4}
 \|w\|_{Y_{h}^{s}}\leq h^{-s+\nu}.
 \end{equation}
 Then
 \begin{equation*}
 T_{h}h^{-p/4}+T_{h}^{\gamma} \|w\|^{p}_{Y_{h}^{s}}\leq h^{3p/4}+h^{p(\gamma-s+\nu)},
 \end{equation*}
 which tends to 0 with $h$, thanks to the assumption made on $\gamma$. Hence, for $h>0$ small enough, with \eqref{bo3} we get 
 \begin{equation*}
 \|w\|_{Y_{h}^{s}}\leq C(T_{h}^\gamma h^{-s})^{1/(1-p)}+CT_{h}h^{-s+1/2-p/4} \leq Ch^{(p\gamma-s)/(1-p)}+Ch^{-s+3p/4+1/2}.
 \end{equation*}
  Finally, observe that $-s+3p/4+1/2>-s+\nu$, and the  assumption $\ds \gamma>s+\frac{\nu(1-p)}{p}$ is equivalent to $(p\gamma-s)/(1-p)>-s+\nu$. Hence  for $h>0$ small enough, we  recover  $\ds  \|w\|_{Y_{h}^{s}}\leq \frac12 h^{-s+\nu}$, and by the usual bootstrap argument,  the condition \eqref{bo4} holds for $T_{h}=h^{p}$. Now  we can deduce the bound $\|u\|_{L^{\infty}_{T}L^{2}}\leq h^{s}\|w\|_{Y_{h}^{s}}\leq h^{\nu}$, which   was the claim.
\\[8pt]
$\bullet$ Case $1<p< 4$. Here we have, by  \cite[Estimate (2.25)]{CaFaHa},  for all $0\leq s<1$
\begin{multline*} 
\big\||w+u_{app}|^{p}(w+{u_{app}})-|u_{app}|^{p}u_{app}\big\|_{H^{s}}\leq \\
C\|u_{app}\|_{H^{s}}(  \|u_{app}\|^{p-1}_{L^{\infty}}+\|w\|^{p-1}_{L^{\infty}}   )\|w\|_{L^{\infty}}+C\|w\|_{H^{s}}\big(\|u_{app}\|^{p}_{L^{\infty}}+\|w\|^{p}_{L^{\infty}}\big).
 \end{multline*}
With the  same arguments as for \eqref{bo1}   we get, with $\widetilde{\gamma}=1-1/r$
\begin{eqnarray*}
\|w\|_{Y^{s}}&\leq& CT^{\widetilde{\gamma}}\|u_{app}\|_{L_{T}^{\infty}H^{s}}\|u_{app}\|^{p-1}_{L_{T}^{\infty}L^{\infty}}\|w\|_{Y^{s}}+CT^{\gamma}\|u_{app}\|_{L_{T}^{\infty}H^{s}} \|w\|^{p}_{Y^{s}}+\\
&&+C\big(T\|u_{app}\|^{p}_{L_{T}^{\infty}L^{\infty}}+T^{\gamma} \|w\|^{p}_{Y^{s}}\big)\|w\|_{Y^{s}}+CTh^{\alpha(p)}\|Q\|_{L^{\infty}_{T} H^{s}}\nonumber\\
&\leq &C\big(  T^{\widetilde{\gamma}}h^{-s-(p-1)/4}+Th^{-p/4}  \big)\|w\|_{Y^{s}}+CT^{\gamma}h^{-s}\|w\|^{p}_{Y^{s}} +CT^{\gamma}\|w\|^{p+1}_{Y^{s}} +CTh^{-s+\alpha(p)},
\end{eqnarray*}
with $\alpha(p)=1-p/2$ when $2\leq p\leq 4$ and   $\alpha(p)=1/2-p/4$ when $1\leq p\leq 2$ (see \eqref{error1} and \eqref{error2}). Then, by the same manner, we get the following a priori estimate  with the semiclassical norm $\|\;\;\|_{Y^{s}_{h}}$ (recall definition \eqref{semic} )
\begin{equation}\label{828}
\|w\|_{Y_{h}^{s}}\leq C\big(  T^{\widetilde{\gamma}}h^{-s-(p-1)/4}+Th^{-p/4}  \big)\|w\|_{Y_{h}^{s}}+CT^{\gamma}h^{-s}\|w\|^{p}_{Y_{h}^{s}} +CT^{\gamma}\|w\|^{p+1}_{Y_{h}^{s}} +CTh^{-s+\alpha(p)}.
\end{equation}
We now perform the bootstrap : Let $r>\max{(2,p)}$ (there will be an additional constraint on $r$). Fix $\ds 1-\frac1r<s<1$ and set $T_{h}=h^{p}$. Then if $r$ is large enough (recall that $\widetilde{\gamma}=1-1/r$), the term $\ds T_{h}^{\widetilde{\gamma}}h^{-s-(p-1)/4}+T_{h}h^{-p/4}$ tends to 0 with $h$, therefore if $h>0$ is small enough, from \eqref{828} we deduce that 
 \begin{equation}\label{829}
\|w\|_{Y_{h}^{s}}\leq  CT_{h}^{\gamma}h^{-s}\|w\|^{p}_{Y_{h}^{s}} +CT_{h}^{\gamma}\|w\|^{p+1}_{Y_{h}^{s}} +CT_{h}h^{-s+\alpha(p)}.
\end{equation}
Choose $0<\nu<p+\alpha(p)$. As previously we assume that 
\begin{equation}\label{pot}
  \|w\|_{Y_{h}^{s}}\leq h^{-s+\nu}.
\end{equation} 
Then with \eqref{829} we get
 \begin{equation*}
\|w\|_{Y_{h}^{s}}\leq   Ch^{\gamma p -s+p(-s+\nu)}+Ch^{\gamma p +(p+1)(-s+\nu)}+Ch^{p-s+\alpha(p)}.
 \end{equation*}
 Next when $r>0$ is large enough (and under the assumption $0<\nu<p+\alpha(p)$), we have $\gamma p -s+p(-s+\nu)>-s+\nu$, $\gamma p +(p+1)(-s+\nu)>-s+\nu$ and $p-s+\alpha(p)>-s+\nu$. To see this, observe that $\gamma\longrightarrow 1$ and $s\longrightarrow 1$ when $r\longrightarrow +\infty$. Therefore for $h>0$ small enough, we  recover  $\ds  \|w\|_{Y_{h}^{s}}\leq \frac12 h^{-s+\nu}$, hence   the condition \eqref{pot} holds for $T_{h}=h^{p}$, and similarly to the previous part,  we deduce that  $\|u\|_{L^{\infty}_{T}L^{2}}\leq h^{s}\|w\|_{Y_{h}^{s}}\leq h^{\nu}$.
 \\[8pt]
$\bullet$ Case $p=1$. By \eqref{828} we have
\begin{equation*}
\|w\|_{Y_{h}^{s}}\leq C\big(  T^{1-\frac1r}h^{-s}+Th^{-1/4}  \big)\|w\|_{Y_{h}^{s}}+CT^{1-\frac1r}\|w\|^{2}_{Y_{h}^{s}} +CTh^{-s+\frac14}.
\end{equation*}
here we set $T_{h}=h^{1+\epsilon}$ with $\epsilon>0$, and we perform the same argument as in the previous case.
\end{proof}
 
Thanks to this proposition and the same argument as \eqref{arg}, we can conclude the proof of Theorem \ref{T:2}.

\appendix

\section{Quasimodes for linear equations near Elliptic Orbits}
\label{quasi}

In this section, we state, without proof, a theorem on existence of
quasimodes near elliptic periodic orbits of the Hamiltonian flow.  A
proof can be found in \cite{Chr-QMNC}.

Let $X$ be a smooth, compact manifold, $\dim X = n$, and suppose $P \in \Psi^{k,0}(X)$, $k
\geq 1$, be a
semiclassical pseudodifferential operator of real principal type which
is semiclassically elliptic outside a
compact subset of $T^*X$.  
Let $\Phi_t = \exp tH_p$ be the classical flow of $p$ and assume there is
a closed {\it elliptic} orbit $\gamma \subset \{ p = 0 \}$.  That $\gamma$
is elliptic means if $N \subset \{ p = 0 \}$ is a Poincar\'e section
for $\gamma$ and $S:N \to S(N)$ is the Poincar\'e map, then $dS(0,0)$ has
eigenvalues all of modulus $1$.  We will also
need the following non-resonance assumption:
\ben
\label{nonres}
\left\{ \begin{array}{l} \text{if } e^{\pm i \alpha_1}, e^{\pm i \alpha_2}, \ldots, e^{\pm i \alpha_k}
 \text{ are eigenvalues of }
dS(0,0), \text{ then } \\ \alpha_1 , \alpha_2, \ldots,
 \alpha_k 
\text{ are independent over }  \pi \ZZ.
\end{array} \right.
\een

Under these assumptions, it is well known that there is a family of
elliptic closed orbits $\gamma_z \subset \{ p = z \}$ for $z$ near
$0$, with $\gamma_0 = \gamma$.  In this work we consider the following eigenvalue problem for $z$ in a
neighbourhood of $z=0$:
\ben
\label{ev-prob-1}
\left\{ \begin{array}{l} (P - z) u = 0; \\ \| u \|_{L^2(X)} =
  1. \end{array} \right.
\een
We prove the following Theorem.
\begin{theorem}
\label{T:quasi-ch}
For each $m \in \ZZ$, $m >1$, and each $c_0>0$ sufficiently small, there is a finite, distinct family of values
\be
\{z_j \}_{j = 1}^{N(h)} \subset [-c_0 h^{1/m}, c_0 h^{1/m} ]
\ee
and a family of quasimodes $\{ u_j\} = \{u_j(h)\}$ with
\be
\WF u_j = \gamma_{z_j},
\ee
satisfying
\ben
\label{ev-prob-2}
\left\{ \begin{array}{l} (P - z_j) u_j = \O(h^\infty) \| u_j \|_{L^2(X)}; \\ \| u_j \|_{L^2(X)} =
  1. \end{array} \right.
\een
Further, for each $m \in \ZZ$, $m >1$, there is a constant $C = C(c_0,{1/m})$ such that
\ben
\label{ev-est-1}
C^{-1} h^{-n(1-1/m)} \leq N(h) \leq C h^{-n}.
\een
\end{theorem}

\begin{rema}
The proof is essentially to construct quasimodes on the Poincar\'e
section as eigenfunctions for the semiclassical harmonic oscillator,
and then to propagate them around the orbit $\Gamma$.  This shows that
the quasimodes have the localization property as in Lemmas
\ref{L:comm-1} and \ref{L:comm-2}.

\end{rema}

\section{Some commutator estimates}

In this section, we prove two results which we have used in the above
computations.  Specifically, we have constructed approximate solutions
to the homogeneous and inhomogeneous equations associated to a
semiclassical operator of the form
\[
Q = hD_t - P(t, x, hD_x),
\]
where $P$ is a ``time-dependent'' harmonic oscillator,
\[
P(t, x, hD_x) = -h^2 \Delta(t) + V,
\]
where
\[
V = b^{ij}(t) x_i x_j
\]
is a positive definite quadratic form.  Both $\Delta(t)$ and $V$ have
$t$-periodic coefficients of period $T$ (the same as $\Gamma$), and we
seek periodic solutions to equations
\[
Q v = E v, \,\,\, (Q -E)v = f,
\]
where $E$ is an eigenvalue to be determined and $f$ is periodic in
$t$.  The constant-coefficient semiclassical harmonic oscillator is
well-known to have eigenfunctions of a semiclassically scaled Hermite
polynomial times a semiclassical Gaussian.  These eigenfunctions, for
small eigenvalues, have semiclassical wavefront set at $(0,0)$.
Moreover, for any $\epsilon, \delta >0$, if $| (x, \xi) | \geq \epsilon
h^{1/2 - \delta}$, these eigenfunctions are $\O(h^\infty)$ in the
Schwartz space.  The purpose of this section is to prove that for the
periodic orbit case, similar localization occurs.

\begin{lemma}
\label{L:comm-1}
Suppose $v$ solves
\[
\begin{cases}
Q v = E v + \O(h^\infty) \| v \|, \,\, E = \O(h), \\
\| v \| = 1, \\
v(0) = v(T) = v_0,
\end{cases}
\]
where $v_0$ satisfies the localization property:
\[
\begin{cases}
\forall \delta, \epsilon>0, \exists \psi \in \Ci_c ( \reals ), \psi
\equiv 1 \text{ on } \{ | (x, \xi ) | \leq \epsilon h^{1/2-\delta} \},
\\
\text{ with } \supp \psi \subset \{ | (x, \xi ) | \leq 2\epsilon
h^{1/2-\delta} \}, \text{ we have } \\
\Op_h( \psi ) v_0 = v_0 + \O(h^\infty ) \| v_0 \|.
\end{cases}
\]
Then there exists $\epsilon_1>0$, $\epsilon_1 \to 0$ as $\epsilon \to
0$, such that for all $0 \leq t \leq T$, if $\psi \in \Ci_c( \reals
)$, $\psi (r) \equiv 1$ for $\{ |r | \leq 1 \}$ with $\supp \psi (r)
\subset \{ | r | \leq 2 \}$, then 
\[
\Op_h ( \psi ( p (t, x, \xi)/ ( \epsilon_1^2
h^{1-2\delta})) ) v(t) = v(t) + \O(h^\infty ) \| v(t) \|.
\]
\end{lemma}

\begin{proof}

Let $\psi (r) \equiv 1$ for $| r | \leq 1$.  Then for any $\epsilon,
\delta>0$,
\[
\Op_h ( \psi(p(0, x, \xi)/(\epsilon^2 h^{1 - 2\delta}))) v_0 = v_0 +
\O(h^\infty),
\]
since $p(0, x, \xi)$ is comparable to $| ( x, \xi ) |^2$.  Let $I(t)$
be the forward propagator for $P$:
\[
\begin{cases}
(hD_t + P(t, x, hD) ) I(t) = 0, \\ 
I(0) = \id_{L^2 \to L^2}.
\end{cases}
\]
Let $\Psi = \Op_h ( \psi(p(0, x, \xi)/(\epsilon^2 h^{1 - 2\delta})))$,
and set 
\[
\Gamma(t) = I(t) \Psi I(t)^{-1}.
\]
We have $hD_t \Gamma = [P(t, x, hD), \Gamma(t, x, hD)]$,
and by Egorov's theorem, $\WF \Gamma$ is contained in the flowout by
$\exp (tH_p)$ of $\{ p(0, x, \xi)\leq 2\epsilon^2 h^{1 - 2\delta}
\}$.  Now the flowout of $\exp(t H_p)$ no longer preserves the level
set of $p$ because $p$ depends on $t$, however, if $(x(t), \xi(t))$ is
an integral curve, then
\[
\frac{d}{dt} p(t, x(t), \xi(t)) = p_t (t, x(t), \xi(t)),
\]
so that there is a constant $C>0$, independent of $h$ so that
\[
-C p \leq \frac{d}{dt} p(t, x(t), \xi(t)) \leq C p,
\]
by the homogeneity of $p$ for $(x, \xi)$ in a neighbourhood of
$(0,0)$.  Hence there is a constant $c_0$ so that
\[
c_0^{-1} p(0, x(0), \xi(0)) \leq p(t, x(t), \xi(t)) \leq c_0 p(0,
x(0), \xi(0))
\]
on the flowout for $0 \leq t \leq T$.  Hence a neighbourhood of
$(0,0)$ of order $h^{1/2 - \delta}$ stays of the same order, although
the size of $\epsilon>0$ may increase.

We have yet to show the asserted identity property acting on $v$.  But
for this, we simply note that $v(t) = I(t) v_0$, and $I(t)^* =
I(t)^{-1}$ to write
\[
\Gamma v(t) = I(t) \Psi v_0 = I(t) (v_0 + \O(h^\infty)_{L^2}) = v(t)  + I(t)
\O(h^\infty)_{L^2},
\]
and since $I(t)$ is unitary, we have proved the Lemma.

\end{proof}

We now know that modes and quasimodes are concentrated on a scale of
$h^{1/2 - \delta}$ for any $\delta>0$.  We can measure the distance to
$\Gamma$ in the transversal direction at the point $t$ on $\Gamma$
using $p(t, x, \xi)$, so it is convenient to have cutoffs which in a
sense depend only on $p$, and more specifically nearly commute with
$P(t, x, hD)$.

\begin{lemma}
\label{L:comm-2}
Let $P(t, x, hD)$ be as above and fix $N>0$.  Fix $\epsilon>0$ sufficiently
small and fix $0 \leq a < b < \epsilon$, and suppose 
$\phi^0 \in \Ci_c( \reals)$ has support $\supp \phi^0 \subset [a,b]$.
Then for each $\delta>0$ and for each $1 \leq j \leq N$, there exist symbols $\phi^j \in
\s_{1/2 - \delta}$ such that
\begin{enumerate}
\item[i] If $\tphi^0(t, x, \xi, h) = \phi^0(p(t, x, \xi) / h^{1 -
    2\delta} )$, then $\tphi^0 \in \s_{1/2 - \delta}$ and $\psi^0 =
  \Op_h ( \tphi^0)$ satisfies
\[
[P, \psi^0] = \O(h^{1/2 + \delta} )_{L^2 \to L^2},
\]

\item[ii] $\supp \phi^j \subset \{ a h^{1/2 - \delta} \leq p \leq b
  h^{1/2 - \delta} \}\cap \supp \nabla \tpsi^0$, for each $1 \leq j
  \leq N$, and 

\item[iii] if $\psi^j = \Op_h( \phi^j)$, then
\[
\psi = \psi^0 + \sum_{j =1 }^N h^{j(1/2 + \delta)} \psi^j
\]
satisfies
\[
[P, \psi] = h^{N(1/2 + \delta)}R,
\]
with $R : L^2 \to L^2$ bounded with compact $h$-wavefront set.
\end{enumerate}
\end{lemma}

\begin{proof}
The proof is relatively standard, however the addition of the periodic
``boundary conditions'' in $t$ adds a small difficulty, so we
reproduce the basic argument here.

The principal symbol of $P$ is
\[
p(t, x, \xi) = a^{ij}(t, x) \xi_i \xi_j + b^{ij}(t) x_i x_j,
\]
with both $a^{ij}$, $b^{ij}$ positive definite matrices.  By the
homogeneity of the quadratic forms, we clearly have $\tphi^0 \in
\s_{1/2 - \delta}$.  The symbol calculus then implies the assertion
(i), once we observe that $\tpsi^0$ cuts off to a compact region, and
hence $p$ is bounded there.

We now proceed to construct the $\phi^j$ satisfying (ii)-(iii).  For
notation simplicity, denote 
\[
A = \{ a h^{1/2 - \delta} \leq p \leq b
  h^{1/2 - \delta} \}\cap \supp \nabla \tpsi^0 \subset T^*M.
\]
From
(i) and the symbol calculus, we have
\[
[P, \psi^0] = \frac{h}{i} \Op_h ( \{ p, \tphi^0 \} ) + R_0,
\]
where $R_0 = h^{1 + 2 \delta} \Op_h(r_0) + R_{0,N}$, where $r_0 \in \s_{1/2 - \delta}$
with $\supp r_0 \subset A$, and $R_{0,N} = \O(h^{N(1/2 +
  \delta)})_{L^2 \to L^2}$ has compact $h$-wavefront set.  We now
compute for an arbitrary choice of $\psi^j \in \s_{1/2 - \delta}$ with compact support:
\[
[P, h^{j(1/2 + \delta)} \psi^j] = h^{j(1/2 + \delta)} \frac{h}{i} \Op_h ( \{
p, \phi^j \} ) + h^{j(1/2 + \delta)} R_j,
\]
where $R_j = h^{1 + 2 \delta} \Op_h (r_j) + R_{j,N}$.  As in the case
of $\phi^0$, here $r_j \in \s_{1/2 - \delta}$ has support contained in
$\supp \nabla \phi^j$ and $R_{j,N} = \O_{L^2 \to L^2}(h^{N(1/2 +
  \delta)})$ has compact $h$-wavefront set.

We observe that $\{ p, \phi^j \}$ has a prefactor of $h^{-1/2 +
  \delta}$ since $\phi^j \in \s_{1/2 - \delta}$ but $p \in \s_0$, so
we want to construct $\phi^j$ so that $h^{j(1/2 + \delta)} h \{ p,
\phi^j \}/i$ cancels the term of order $h^{(j+1)(1/2 + \delta)}$.  
Assume for $1 \leq k \leq
j-1$ we have found $\phi^k$ satisfying (ii) such that
\[
i h^{1/2 - \delta} \{ p, \phi^k \} - r_{k-1} = \O(h^{1/2 - \delta})
\]
with the $r_{k-1}$ and the error in $\s_{1/2 - \delta}$ with support
in $A$.  Then if $\Gamma^{j-1} = \sum_{k = 0}^{j-1}h^{k(1/2 + \delta)}
\psi^k$ satisfies
\[
[P, \Gamma^{j-1} ] = h^{(j-1)(1/2 + \delta)} \tilde{R}_{j-1},
\]
where $\tilde{R}_{j-1} = h^{1 + 2 \delta } \Op_h (\tilde{r}_{j-1} ) + R_{j-1,N}$
with $r_{j-1} \in \s_{1/2 - \delta}$, $\supp r_{j-1} \subset A$, and
$R_{j-1,N} = \O(h^{N(1/2 + \delta)})$ with compact $h$-wavefront set.
Then we want to solve
\[
i h^{1/2 - \delta} \{ p, \phi^j \} - r_{j-1} = \O(h^{1/2 -\delta}).
\]
If $r_{j-1,0}$ is the principal symbol of $r_{j-1}$, we apply the
Frobenius theorem to find such a $\phi^j$.  The support properties
follow from the assumed support properties on $r_{j-1}$ and the
observation that we can always multiply $\phi^j$ by a function of $p$
to ensure it is supported in $A$, at the expense of another compactly
supported error of order $h^{1/2 + \delta}$.  The symbolic properties
follow from the assumed symbolic properties of $r_{j-1}$.

\end{proof}

\section{Real principal type}

Let us quickly show that our local principal symbol near $\Gamma$ can
be glued into a symbol of real principal type.  Let us recall that a
symbol $p$ is of 
real principal type if the principal symbol is real valued, smooth,
has compact level sets, is elliptic outside a compact set, and $dp
\neq 0$ on $\{p = 0 \}$.  Our
local model near $\Gamma$ is of the form
\[
p_1(t, x, \xi) = - \tau + a^{ij}(t, x) \xi_i \xi_j + b^{ij}(t) x_i
x_j,
\]
$a^{ij}$ and $b^{ij}$ are positive definite symmetric matrices, $t \in \SS^1$ and $\tau$ is the dual variable to $t$.  Let $p_0
= a^{ij}(t, x) \xi_i \xi_j + b^{ij}(t) x_i
x_j$, which is a suitable measure of the distance squared to $\Gamma$.
Fix $\delta>0$, and let $\chi(r) \in \Ci_c(\reals)$ be equal to $1$
for $| r | \leq \delta$, and $\chi(r) \equiv 0$ for $| r | \geq 2
\delta$.  Let $q = \tau^2 + p_0$, which is elliptic outside a
compact set, and set
\[
p_2 = \chi(p_0) p_1 + (1 - \chi(p_0)) q.
\]
The function $p_2$ satisfies all the requirements of real principal
type, once we show that $d p_2 \neq 0$ on $\{ p_2 = 0 \}$.  For this,
first note that
\[
\{ p_2 = 0 \} \subset ( \{ \tau = p_0 \} \cap \{ p_0 \leq 2\delta \} )
\cup (\{ q = 0 \} \cap \{ p_0 \geq \delta\}).
\]
The latter set is empty, but the first is not.
We observe that 
\[
dp_2 = (1 + \chi'(p_0)(- \tau -  \tau^2)) d p_0 + (2(1 - \chi) \tau -
\chi) d \tau,
\]
and that $dp_0 = 0$ only if $x = \xi = 0$.  
Now $\chi$ is a non-increasing function of $p_0$, so on $\{ \tau = p_0
\} \cap \{ p_0 \leq 2\delta \}$, $\chi'(p_0) (- \tau - \tau^2) \geq
0$.  Hence if $(x, \xi) \neq (0,0)$, $dp_2 \neq 0$.  If $x = \xi = 0$,
we have $p_0 = 0$, so on $\{\tau = p_0 \} \cap \{ p_0 \leq 2\delta
\}$, $\tau = 0$ also.  But 
\[
(2(1 - \chi) \tau -
\chi) d \tau \neq 0
\]
for $\tau$ in a neighbourhood of $0$, which shows $d p_2 \neq 0$ on
$\{ p_2 = 0 \}$.

\section{Rescaled wavefront sets: an example}

In this section, we provide for the reader's convenience an example
where one encounters 
the rescaled wavefront sets.  Let us consider the quantum harmonic
oscillator
\[
P = (hD_x)^2 + x^2.
\]
The eigenfunctions $P u_j = E_j u_j$, are semiclassical Hermite
polynomials times semiclassical Gaussians.  If $u_0(x) = c_0 h^{-1/4}
e^{-x^2/2h}$ is the ($L^2$-normalized) semiclassical Gaussian, and if
$\chi(x) \in \Ci_c(\reals)$ is equal to $1$ near $0$, then clearly
\[
\| (1 - \chi(h^{-\delta} x)) u_0 \| = \O(h^\infty)
\]
in any seminorm, provided $\delta<1/2$.  On the other hand, a simple
computation shows that the semiclassical Fourier transform of $u_0$
\[
\mathcal{F}_h(u_0)( \xi) = c_0' h^{-1/4} e^{-\xi^2/2h},
\]
so that $\mathcal{F}_h(u_0)$ has the same localization in $\xi$ as
$u_0$ did in $x$.  In other words, for any $0 \leq \delta < 1/2$, 
\[
\dWF(u_0) = \{(0,0) \}.
\]
A similar argument implies the same is true for any $u_j$ provided
$E_j = \O(h)$.  

On the other hand, let us see how to use Corollary \ref{C:dWF-cor} to
prove the same localization.  If $E_j = \O(h)$, then the principal of
$P - E_j$ is $p = \xi^2 + x^2$.  By homogeneity, we can rescale 
\[
p = h^{2 \delta} ( ( h^{-\delta}\xi)^2 + (h^{-\delta} x)^2),
\]
which is a symbol in $\s^{2,-2 \delta}_\delta$.  As this symbol is
$\O(h^\infty)$ only at $(x, \xi) = (0,0)$, we get the same result as
above.

\section{Harmonic oscillator eigenfunctions}
\label{A:harmosc}

Let $P_0 = -d^2/dx^2 + x^2$ be the one-dimensional quantum harmonic
oscillator operator.  The theory of the eigenfunctions of $P_0$ is
well established, however we need several important facts recalled
here.

\begin{proposition}
Each eigenfunction $u_n$, $n = 0, 1, 2, \ldots$, of $P_0$ is a polynomial times a
Gaussian:
\[
u_n(x) = H_n(x) e^{-x^2/2},
\]
with eigenvalue $\lambda_n = 2n+1$.

The $u_n$ can be taken to be real-valued, they form a complete
orthonormal basis of $L^2( \reals)$, and the zeros of $u_n$ are simple.

\end{proposition}

\begin{proof}
The proof as usual is by creation and annihiliation operators.  Let 
\[
A_\pm = D_x \pm i x,
\]
so that $A_\pm = A_\mp^*$ and
\[
A_+ A_- = P_0 - 1 = A_- A_+ - 2.
\]
A computation shows
\[
A_- e^{-x^2/2} = 0,
\]
so that
\[
A_+ A_- e^{-x^2/2} = 0 = (P_0 -1) e^{-x^2/2}.
\]
The other eigenfunctions are constructed with the creation operator:
\[
u_n = A_+^n e^{-x^2/2},
\]
and a simple computation shows the $u_n$ is a polynomial times the
Gaussian $e^{-x^2/2}$.  The creation operator $A_+ = -i( \partial_x
-x)$ is $-i$ times a real-valued operator, so the polynomials can be
taken to be real-valued.  By induction, we have
\begin{align*}
P_0 u_n & = (A_+ A_- +1) u_n \\
& = A_+ (A_- A_+^n) u_0 + u_n \\
& = A_+(P_0 + 1) A_+^{n-1} u_0 + u_n \\
& = A_+ (2(n-1) + 2) u_{n-1} + u_n \\
& = (2n+1)u_n.
\end{align*}
A simple computation shows the eigenfunctions $\{u_n \}$ form a
complete orthnormal set.

To show the zeros are simple, we again assume for induction that $u_n$
has only simple zeros, and 
suppose $u_{n+1}(x_0) = 0$.  Then
\begin{align*}
A_- u_{n+1} & = A_- A_+ u_n \\
& = (P_0 + 1) u_n \\
& = (2n+2) u_n,
\end{align*}
so 
\[
u_{n+1}' (x) + x u_{n+1}(x) = i(2n+2) u_n(x),
\]
and if $u_{n+1}(x_0) = u_{n+1}'(x_0) = 0$, then $x_0$ is a zero of
$u_n$ as well.  Differentiating again, and using the fact that
$u_{n+1}$ is an eigenfunction, we get
\begin{align*}
i (2n+2) u_n'(x) & = u_{n+1}''(x) + u_{n+1}(x) + x u_{n+1}'(x) \\
& = - (2n+2) u_{n+1}(x) + x^2 u_{n+1}(x) + x u_{n+1}'(x) ,
\end{align*}
so if $u_{n+1}(x_0) = u_{n+1}'(x_0) = 0$, then $u_n'(x_0) = 0$ as
well, which contradicts the induction hypothesis.  Hence $u_{n+1}(x)$
has only simple zeros.

\end{proof}

We are interested in these properties of the quantum harmonic
oscillator eigenfunctions because, for the nonlinear problem studied in
this note, we will take a non-smooth function of these eigenfunctions,
and we want to understand the singularities.  Let $p>0$, let $u_n(x)$
be an eigenfunction of the quantum harmonic oscillator, and set
\[
v(x) = |u_n(x)|^p u_n(x).
\]

\begin{proposition}
\label{P:v-symbolic}
The function $v(x)$ is rapidly decaying, $v(x) \in \mathcal{C}^1 \cap
H^1$, and 
\[
\hat{v}( \xi ) \in \scl^{-2-p},
\]
where $\scl^{-2-p}$ is the space of classical symbols of order $-2-p$.

\end{proposition}

In particular, we are interested in semiclassical rescaling, and to
what extent $v$ is localized in phase space.  The function $v$ is not
smooth, so it does not have compact semiclassical wavefront set, but
because of the symbolic assertion in the previous proposition, there
is some decay at infinity, as described in the next corollary.

\begin{corollary}
\label{C:v-symbolic}
For any $\delta>0$, the function $v(x)$ satisfies:
\[
\| v(x) \|_{H^1(B(0, h^{-\delta})^\complement)} = \O(h^\infty).
\]
Moreover, for any $0 \leq \gamma \leq 1$ and $0 \leq s \leq 3/2$, the semiclassical Fourier transform satisfies
\[
\| | \xi/h|^s \F_h v \|_{L^2(B(0, h^\gamma)^\complement)} \leq C h^{(1
  - \gamma)(3/2 + p -s )}.
\]

In particular, if $\chi(x, \xi) \in \Ci_c(\reals^{2})$ is $1$ in a
neighbourhood of $(0,0)$, then
\[
 \chi(h^\delta x, h^{1-\gamma} D_x) v = v + E,
\]
where for any $0 \leq s \leq 3/2$, 
\[
\| E \|_{\dH^s} \leq C h^{(1
  - \gamma)(3/2 + p -s )}.
\]

\end{corollary}

\begin{proof}[Proof of Proposition \ref{P:v-symbolic}]

Each zero of $u_n$ is a simple zero, so we can write 
\[
v(x) = \sum_{l=0}^n v_l(x),
\]
with $v_0 \in \s$, and for $1 \leq l \leq n$, $v_l$ has compact
support containing a single zero of $u_n$.  If $x_l$ is a zero of
$u_n$ contained in the support of $v_l$, then 
\[
v_l ( x + x_l) \sim | x |^p x
\]
near $x = 0$, and $v_l(x + x_l)$ is smooth and compactly supported
away from from $x=0$.  Then $v_l(x + x_l)$ is conormal at $x = 0$,
which implies $\hat{v}_l$ is a symbol in class $\scl^{-2-p}$.  Summing
in $l$, and using that $\hat{v}_0 \in \scl^{-\infty}$, we get $\hat{v}
\in \scl^{-2-p}$ as claimed.

\end{proof}


\bibliographystyle{alpha}
\bibliography{acgmt-bib}

\end{document}